\documentclass[11pt,oneside,reqno]{article}
\usepackage[margin=1.3in]{geometry}

\usepackage{amsmath,amsthm,amssymb,color}
\usepackage[authoryear]{natbib}
\usepackage{booktabs}

\usepackage{bbm}
\usepackage{mathrsfs}
\usepackage[breaklinks=true]{hyperref}
\usepackage{graphicx}
\usepackage{tikz}
\usepackage{enumitem}
\usetikzlibrary{arrows, automata}
\usetikzlibrary{calc}
\usetikzlibrary{positioning}

\numberwithin{equation}{section}
\allowdisplaybreaks[4]

\theoremstyle{plain}
\newtheorem{theorem}{Theorem}[section]
\newtheorem{proposition}{Proposition}[section]
\newtheorem{lemma}{Lemma}[section]
\newtheorem{corollary}{Corollary}[section]

\theoremstyle{definition}

\newtheorem{remark}{Remark}[section]

\makeatletter

\newcount\minute
\newcount\hour
\newcount\hourMins
\def\now{%
\minute=\time%
\hour=\time \divide \hour by 60%
\hourMins=\hour \multiply\hourMins by 60%
\advance\minute by -\hourMins%
\zeroPadTwo{\the\hour}:\zeroPadTwo{\the\minute}%
}
\def\zeroPadTwo#1{\ifnum #1<10 0\fi#1}

\renewcommand{\cite}{\citet}

\def\^#1{\ifmmode {\mathaccent"705E #1} \else {\accent94 #1} \fi}
\def\~#1{\ifmmode {\mathaccent"707E #1} \else {\accent"7E #1} \fi}

\def\*#1{#1^\ast}
\edef\-#1{\noexpand\ifmmode {\noexpand\bar{#1}} \noexpand\else \-#1\noexpand\fi}
\def\>#1{\vec{#1}}
\def\.#1{\dot{#1}}
\def\wh#1{\widehat{#1}}
\def\wt#1{\widetilde{#1}}

\def\atop{\@@atop}
\def\*#1{\mathscr{#1}}

\renewcommand{\leq}{\leqslant}
\renewcommand{\geq}{\geqslant}
\newcommand{\eps}{\varepsilon}
\renewcommand{\eps}{\varepsilon}

\newcommand{\eq}{\eqref}

\newcommand{\IE}{\mathbbm{E}}
\newcommand{\IP}{\mathbbm{P}}
\newcommand{\Var}{\mathop{\mathrm{Var}}\nolimits}

\newcommand{\tr}{\mathop{\mathrm{tr}}}

\newcommand{\IR}{\mathbb{R}}

\def\be#1{\begin{equation*}#1\end{equation*}}
\def\ben#1{\begin{equation}#1\end{equation}}
\def\bes#1{\begin{equation*}\begin{split}#1\end{split}\end{equation*}}
\def\besn#1{\begin{equation}\begin{split}#1\end{split}\end{equation}}

\def\bm#1{\begin{multline*}#1\end{multline*}}
\def\bmn#1{\begin{multline}#1\end{multline}}
\def\ba#1{\begin{align*}#1\end{align*}}
\def\ban#1{\begin{align}#1\end{align}}

\def\bkle#1{\bigl[#1\bigr]}

\def\bklg#1{\bigl\{#1\bigr\}}

\def\babs#1{\bigl\vert#1\bigr\vert}

\def\mid{\vert}

\def\beqn#1\eeqn{\begin{align}#1\end{align}}
\def\beq#1\eeq{\begin{align*}#1\end{align*}}

\usepackage{graphicx}
\usepackage{latexsym}
\usepackage{amsmath,amsthm,amssymb,amscd}
\usepackage{epsf,amsmath}

\def\E{{\IE}}
\def\P{{\IP}}

\newcommand{\mr}{\IR^d}

\newcommand{\mcl}[1]{\mathcal{#1}}

\DeclareMathOperator{\Hess}{Hess}

\renewcommand\section{\@startsection {section}{1}{\z@}%
{-3.5ex \@plus -1ex \@minus -.2ex}%
{1.3ex \@plus.2ex}%
{\center\small\sc\mathversion{bold}\MakeUppercase}}

\def\subsection#1{\@startsection {subsection}{2}{0pt}%
{-3.5ex \@plus -1ex \@minus -.2ex}%
{1ex \@plus.2ex}%
{\bf\mathversion{bold}}{#1}}

\def\subsubsection#1{\@startsection{subsubsection}{3}{0pt}%
{\medskipamount}%
{-10pt}%
{\normalsize\itshape}{\kern-2.2ex. #1.}}

\def\blfootnote{\xdef\@thefnmark{}\@footnotetext}

\makeatother

\begin{document}

\title{Large-dimensional Central Limit Theorem with Fourth-moment Error Bounds on Convex Sets and Balls}
\author{Xiao Fang and Yuta Koike}
\date{\it The Chinese University of Hong Kong and The University of Tokyo} 
\maketitle

\noindent{\bf Abstract:} 
We prove the large-dimensional Gaussian approximation of a sum of $n$ independent random vectors in $\mathbb{R}^d$ together with fourth-moment error bounds on convex sets and Euclidean balls. We show that compared with classical third-moment bounds, our bounds have near-optimal dependence on $n$ and can achieve improved dependence on the dimension $d$.
For centered balls, we obtain an additional error bound that has a sub-optimal dependence on $n$, but recovers the known result of the validity of the Gaussian approximation if and only if $d=o(n)$.
We discuss an application to the bootstrap. 
We prove our main results using Stein's method.

\medskip



\section{Introduction}

Let $\{\xi_i\}_{i=1}^n$ be a sequence of independent mean-zero random vectors in $\mathbb{R}^d$.
Let $W=\sum_{i=1}^n \xi_i$ and $\Sigma=\Var(W)$.  
It is well known that under finite third-moment conditions and for fixed dimension $d$, the distribution of $W$ can be approximated by a Gaussian distribution with error rate $O(1/\sqrt{n})$.

Motivated by modern statistical applications, we are interested in the large-dimensional setting where $d$ grows with $n$.
Numerous studies have provided explicit error bounds on various distributional distances in the Gaussian approximation.
See, for example, \cite{Be03,Be05} and \cite{Ra19a} for results for the probabilities of convex sets in $\mathbb{R}^d$; \cite{CCK13, CCK17}, \cite{CCKK19}, \cite{FaKo20a}, \cite{Lo20} and \cite{KuRi20} for results for hyperrectangles; and \cite{Zh18}, \cite{EMZ18}, \cite{Ra19b} and \cite{Bo20} for results for the Wasserstein distance in the approximation.
However, the optimal rates, especially in terms of how rapidly $d$ can grow with $n$ while maintaining the validity of the Gaussian approximation, have not been fully addressed and remain a challenging open problem.

In this paper, we consider the approximation of probabilities of convex sets and Euclidean balls. 
For convex sets, \cite{Be05} proved for the above $W$ that if $\Sigma$ is invertible and $Z\sim N(0,\Sigma)$, then
\ben{\label{n23}
\sup_{A\in \mathcal{A}} |\P(W\in A)-\P(Z\in A)|\leq C d^{1/4} \sum_{i=1}^n \E[|\Sigma^{-1/2}\xi_i|^3 ],
}
where $\mathcal{A}$ is the collection of all measurable convex sets in $\IR^d$, $C$ is an absolute constant and $|\cdot|$ denotes the Euclidean norm when applied to a vector.
\cite{Ra19a} obtained an explicit constant in the error bound \eq{n23}.
The error bound \eq{n23} is optimal up to the factor $d^{1/4}$ because, as shown by \cite{Na76}, the bound no longer holds if we replace $d^{1/4}$ by any vanishing quantity.
For Euclidean balls, it is known that the factor $d^{1/4}$ can be removed if $\Sigma=I_d$, the $d\times d$ identity matrix. This was proved in \cite{Be03} for the independent and identically distributed (i.i.d.) case. The general case follows from \cite[Theorem 1.3 and Example 1.2]{Ra19a} and \cite[Remark 2.1]{Sa72}, for example.

Our first main result (cf. Theorem \ref{convex}) is that up to a logarithmic factor,
\ben{\label{n24}
\sup_{A\in \mathcal{A}} |\P(W\in A)-\P(Z\in A)|\leq_{\log} C d^{1/4} \left(\sum_{i=1}^n \E[|\Sigma^{-1/2}\xi_i|^4]\right)^{1/2}.
}
The bound \eq{n24} is optimal up to the $d^{1/4}$ and the logarithmic factors (cf. Proposition \ref{prop2}).
We will argue that (cf. Remark \ref{rem1}) under finite fourth-moment conditions, the bound \eq{n24} has near-optimal dependence on $n$. Moreover, perhaps surprisingly, it can achieve better dependence on dimension compared with \eq{n23}.
We note that applying the Cauchy-Schwarz inequality to \eq{n23} results in a bound as in \eq{n24}, but with an additional factor of $d^{1/2}$. It is the removal of this factor that enables the improvement of the dependence on dimension.

We then consider the Gaussian approximation on the class $\mathcal{B}$ of all Euclidean balls, which is arguably most relevant for statistical applications, e.g., chi-square tests. 
We show that (cf. Theorem \ref{BALL}) the factor $d^{1/4}$ in \eq{n24} can be removed if we replace $\mathcal{A}$ and $|\Sigma^{-1/2}\xi_i|$ with $\mathcal{B}$ and $\|\Sigma^{-1/2}\|_{op}|\xi_i|$, respectively. 
Furthermore, for centered balls, we obtain an additional error bound (cf.~Theorem \ref{BALL2}) that has a sub-optimal dependence on $n$, but recovers the known result of the validity of the Gaussian approximation as long as $d=o(n)$. 
Incidentally, the requirement $d=o(n)$ is necessary for the validity of the Gaussian approximation on balls (cf. Proposition \ref{prop1}).

We prove our main results using Stein's method (\cite{St72}) and its recent advances. To prove \eq{n24}, we use a Gaussian anti-concentration inequality for convex sets by \cite{Ba93}, the recursive argument of \cite{Ra19a}, a multivariate exchangeable pair coupling (\cite{ChMe08} and \cite{ReRo09}) and a symmetry argument in \cite{FaKo20a, FaKo20b}. 
To prove the results for balls, we further use Gaussian anti-concentration inequalities for ellipsoids by \cite{Zh20} and \cite{GNSU19}.

The bound \eq{n23} and its variants have been widely used in the statistics literature, especially in inference for models with large parameter dimensions. 
See, for example, \cite{SpZh15}, \cite{Po15}, \cite{PeSc18}, \cite{Sh19} and \cite{ChZh20}. 
Our new bounds' improved dependence on dimension may prove useful if we are interested in allowing $d$ to grow as rapidly as possible depending on $n$, which is one of the most important subjects in such literature. 
We will also discuss an application to the bootstrap that is ubiquitous in this field (see Section \ref{sec4}).

The paper is organized as follows.
In Sections \ref{sec2} and \ref{sec3}, we state our main results for the large-dimensional Gaussian approximation of sums of independent random vectors on convex sets and balls, respectively. In Section \ref{sec4}, we discuss an application to the bootstrap. Section \ref{sec5} contains all of the proofs.

For a matrix $M$, we use $\|M\|_{op}$ and $\|M\|_{H.S.}$ to denote its operator norm and Hilbert-Schmidt norm, respectively. We use $C$ to denote positive absolute constants which may differ in different expressions. For a vector $x\in \mathbb{R}^d$, we use $x_j, 1\leq j\leq d$ to denote its components. For a sequence of vectors $x_i\in \mathbb{R}^d, 1\leq i\leq n$, we use $x_{ij}$ to denote the $j$th component of $x_i$ for $1\leq j\leq d$. Similarly, we write $X_j$ and $X_{ij}$ for the components of random vectors $X$ and $X_i$, respectively.

\section{Approximation on Convex Sets}\label{sec2}

In this section, we consider the Gaussian approximation of sums of independent random vectors on convex sets. Our main result is the following fourth-moment error bound in the approximation.


\begin{theorem}\label{convex}
Let $\xi=\{\xi_i\}_{i=1}^n$ be a sequence of centered independent random vectors in $\IR^d$ with finite fourth moments and set $W=\sum_{i=1}^n\xi_i$. 
Let $Z\sim N(0,\Sigma)$ be a centered Gaussian vector in $\IR^d$ with covariance matrix $\Sigma$.
Assume $\Sigma$ is invertible.
Then,
\ben{\label{n01}
\sup_{A\in \mathcal{A}} |\P(W\in A)-\P(Z\in A)|\leq C d^{1/4} \Psi\left(\delta_\mathcal{A}(W,\Sigma)\right),
}
where
$\mathcal{A}$ is the collection of all measurable convex sets in $\IR^d$, $\Psi(x)=x(|\log x|\vee1)$ and
\[
\delta_\mathcal{A}(W,\Sigma):=\|I_d-\Var(\Sigma^{-1/2} W)\|_{H.S.}
+\left(\sum_{i=1}^n \E [|\Sigma^{-1/2} \xi_i|^4]\right)^{1/2}.
\]
\end{theorem}

Note that we do not assume $W$ and $Z$ have exactly the same covariance matrix. This facilitates potential applications to the bootstrap approximation and also helps in the recursive argument in the proof, where we need to consider approximating $W-\xi_i$ by $Z$.

The next result shows that the bound \eq{n01} is optimal up to the $d^{1/4}$ and the logarithmic factors.

\begin{proposition}\label{prop2}
There is an absolute constant $C_0>0$ such that, for sufficiently large $n$ $d\leq\sqrt{n}/\log n$, we can construct centered i.i.d.~random vectors $\xi_1,\dots,\xi_n$ in $\mathbb{R}^d$ with finite fourth moments (which may depend on $n$ and $d$) satisfying $\Var(W)=I_d$ and
\be{
\sup_{A\in \mathcal{A}} |\P(W\in A)-\P(Z\in A)|\geq C_0 \left(\sum_{i=1}^n \E[|\xi_i|^4]\right)^{1/2},
}
where $Z\sim N(0,I_d)$.
\end{proposition}

We use the next remark to discuss the crucial fact that our bound \eq{n01} may be preferable to the third-moment bound \eq{n23} in the large-dimensional setting.

\begin{remark}\label{rem1}
To understand the typical order of the right-hand side of \eq{n01}, we consider the situation where $\xi_i=X_i/\sqrt{n}$ and $\{X_1,X_2,\dots\}$ is a sequence of i.i.d.~mean-zero random vectors in $\mathbb{R}^d$ with $\Var(X_i)=I_d$. Let $\Sigma=I_d$.
For the $d$-vector $X_i$, $\E [|X_i|^3]$ and $\E [|X_i|^4]$ are typically proportional to $d^{3/2}$ and $d^{2}$, respectively. 
In this case, the right-hand side of \eq{n01} is of the order $O(\frac{d^{5/2}}{n})^{1/2}$ up to a logarithmic factor. In contrast, the right-hand side of \eq{n23} is of the order $O(\frac{d^{7/2}}{n})^{1/2}$.
Therefore, subject to the requirement of the existence of the fourth moment, \eq{n01} is preferable to \eq{n23} in the large-dimensional setting where $d\to \infty$.
We mention in this context that \cite[Corollary 1.5]{Zh18} obtained a bound typically of the order $O(\frac{d^{5/2}}{n})^{1/3}$ up to a logarithmic factor under a boundedness condition. He obtained the bound as a by-product of a Wasserstein-2 bound in the Gaussian approximation.
\end{remark}


\section{Approximation on Balls}\label{sec3}

In this section, we consider the Gaussian approximation of sums of independent random vectors on Euclidean balls.
In line with the results of Bentkus (2003) and Rai\v{c} (2019a) for the third-moment bound,
our first result shows that the factor $d^{1/4}$ appearing on the right-hand side of \eq{n01} may be removed if we restrict the approximation to the class of balls.
Again, we do not assume $W$ and $Z$ have the same covariance matrix.

\begin{theorem}\label{BALL}
Let $\xi=\{\xi_i\}_{i=1}^n$ be a sequence of centered independent random vectors in $\IR^d$ with finite fourth moments and set $W=\sum_{i=1}^n\xi_i$. 
Let $Z\sim N(0,\Sigma)$ be a centered Gaussian vector in $\IR^d$ with covariance matrix $\Sigma$.
Assume $\Sigma$ is invertible. 
Then 
\ben{\label{ball1}
\sup_{A\in\mcl{B}}|\P(W\in A)-\P(Z\in A)|\leq C\Psi\left(\delta_\mathcal{B}(W,\Sigma)\right),
}
where $\mcl{B}$ is the set of all Euclidean balls in $\IR^d$, $\Psi(x)=x(|\log x|\vee1)$ and
\[
\delta_\mathcal{B}(W,\Sigma):=\|\Sigma^{-1}\|_{op}\left\{\|\Sigma-\Var(W)\|_{H.S.}
+\left(\sum_{i=1}^n \E [|\xi_i|^4]\right)^{1/2}\right\}.
\]
\end{theorem}

Following Remark \ref{rem1}, we can see that if $\Var(W)=\Sigma$, then the typical order of the right-hand side of \eq{ball1} is $O(\frac{d^2}{n})^{1/2}$ up to a logarithmic factor. It has near-optimal dependence on $n$ and converges to 0 if $d=o(\sqrt{n})$.
It remains an open problem whether the growth rate $d=o(\sqrt{n})$ is optimal or not.

In the next result, we sacrifice the rate of $n$ to obtain the optimal growth rate of $d=o(n)$ in terms of the dimension for \textit{centered} balls (cf. Corollary \ref{cor3} and Proposition \ref{prop1} below). 
Similar results have been obtained in the literature (cf. Remark \ref{rem:1}). 
Our main contribution here is a new proof using Stein's method, which works beyond sums of independent random vectors. See the Appendix for a result for sequences of $m$-dependent random vectors.
Compared with the proof of Theorem \ref{BALL}, we use a new smoothing of the indicator function of centered balls in the proof of Theorem \ref{BALL2}.
As a result of the new smoothing, the error bound involves $\|\Var(W)\|_{op}$, which we have to take extra care in the application to bootstrap.
However, we allow $\Sigma$ to be singular.


Given a $d\times d$ symmetric matrix $A$, we denote its eigenvalues arranged in decreasing order by $\lambda_1(A)\geq\cdots\geq\lambda_d(A)$. Then we set
\[
\Lambda_k(A):=\sqrt{\sum_{j=k}^d\lambda_j(A)^2}\qquad\text{for }k=1,2.
\]
When $\Lambda_2(A)>0$, we define $\varkappa(A):=(\Lambda_1(A)\Lambda_2(A))^{-1/2}$. This quantity appears when we apply a Gaussian anti-concentration inequality of \cite{GNSU19} (see Lemma \ref{anti-ball2}).  
\begin{theorem}\label{BALL2}
Let $\xi$, $W$ and $Z$ be as in Theorem \ref{BALL}. 
Set $\Sigma_W=\Var(W)$ and assume $\Lambda_2(\Sigma)>0$ (instead of assuming $\Sigma$ is invertible). Then
\bmn{\label{ball2}
\sup_{r\geq0}|\P(|W|\leq r)-\P(|Z|\leq r)|
\leq C(\varkappa^{3/4}(\Sigma)\delta_1^{1/4}(W)+\varkappa^{2/3}(\Sigma)\delta_2^{1/3}(W)\\
+\varkappa^{2/3}(\Sigma)\delta_0^{1/3}(\Sigma_W,\Sigma)+\varkappa^{1/2}(\Sigma)\delta_0'^{1/2}(\Sigma_W,\Sigma)),
}
where
\ba{
\delta_0(\Sigma_W,\Sigma)&:=\sqrt{(\tr(\Sigma_W)+\tr(\Sigma))(\|\Sigma_W\|_{op}+\|\Sigma\|_{op})}\|\Sigma-\Sigma_W\|_{H.S.},\\
\delta'_0(\Sigma_W,\Sigma)&:=\sum_{j=1}^d|\Sigma_{jj}-\Sigma_{W,jj}|,\\
\delta_1(W)&:=\|\Sigma_W\|_{H.S.}\sum_{i=1}^n\E[|\xi_i|^4]
+\|\Sigma_W\|_{op}^{3/2}\sum_{i=1}^n\E[|\xi_i|^{3}],\\
\delta_2(W)&:=\|\Sigma_W\|_{op}^{1/2}\sum_{i=1}^n\E[|\xi_i|^3]
+\sum_{i=1}^n\E[|\xi_i|^4].
} 
\end{theorem}

\begin{corollary}\label{cor3}
Let $\{X_i\}_{i=1}^n$ be a sequence of centered independent random vectors in $\mathbb{R}^d$.
Let $W=\frac{1}{\sqrt{n}} \sum_{i=1}^n X_i$. Suppose $\Var(W)=I_d$ and $\max_{1\leq i\leq n}\max_{1\leq j\leq d}\E[|X_{ij}|^4]\leq C$. Let $Z\sim N(0, I_d)$. Then
\ben{\label{n42}
\sup_{r\geq0}|\P(|W|\leq r)-\P(|Z|\leq r)|\leq C\left\{ \frac{1}{n^{1/8}}+\left(\frac{d}{n}\right)^{1/6}\right\}.
}
\end{corollary}

\begin{proof}[Proof of Corollary \ref{cor3}]
Let $\xi_i=X_i/\sqrt{n}$.
Then we have for any $r\geq2$
\[
\sum_{i=1}^n\E[|\xi_i|^r]\leq \frac{d^{r/2}}{n^{r/2-1}}\max_{1\leq i\leq n}\max_{1\leq j\leq d}\E[|X_{ij}|^r]
\]
by the Jensen inequality. Therefore, under the condition $\Sigma_W=\Sigma=I_d$, $\delta_0(\Sigma_W,\Sigma)=\delta'_0(\Sigma_W,\Sigma)=0$ and 
\ba{
\delta_1(W)&\leq \frac{d^{5/2}}{n}\max_{1\leq i\leq n}\max_{1\leq j\leq d}\E[|X_{ij}|^4]
+\frac{d^{3/2}}{n^{1/2}}\max_{1\leq i\leq n}\max_{1\leq j\leq d}\E[|X_{ij}|^3],\\
\delta_2(W)&\leq \frac{d^{3/2}}{n^{1/2}}\max_{1\leq i\leq n}\max_{1\leq j\leq d}\E[|X_{ij}|^3]
+\frac{d^{2}}{n}\max_{1\leq i\leq n}\max_{1\leq j\leq d}\E[|X_{ij}|^4].
}
Also, $\varkappa(\Sigma)\leq Cd^{-1/2}$. Consequently, under the condition $\max_{1\leq i\leq n}\max_{1\leq j\leq d}\E[|X_{ij}|^4]\leq C$, the right hand side of \eqref{ball2} is bounded by
\[
C\left\{\left(\frac{d}{n}\right)^{1/4}+\frac{1}{n^{1/8}}
+\left(\frac{d}{n}\right)^{1/6}+\left(\frac{d}{n}\right)^{1/3}\right\}.
\]
\end{proof}

The bound in Corollary \ref{cor3} converges to 0 as long as $d/n\to0$. 
It is not difficult to prove this condition is generally necessary for convergence of the quantity on the left-hand side of \eqref{n42}:

\begin{proposition}\label{prop1}
Let $X_1,\dots,X_n$ be i.i.d.~standard Gaussian vectors in $\mathbb{R}^d$. Let $\{e_i\}_{i=1}^{n}$ be i.i.d.~variables independent of $\{X_i\}_{i=1}^n$ with $\E e_1=0$, $\E [e_1^2]=1$, $\E [e_1^4]<\infty$ and $\Var(e_1^2)>0$. 
Assume the law of $e_1$ does not depend on $n$.
Set $W:=n^{-1/2}\sum_{i=1}^ne_iX_i$ and let $Z\sim N(0,I_d)$. If
\ben{\label{n41}
\sup_{r\geq0}|\P(|W|\leq r)-\P(|Z|\leq r)|\to0
}
as $d,n\to\infty$, we must have $d/n\to0$. 
\end{proposition}

\begin{remark}
$W$ in Proposition \ref{prop1} can be regarded as a bootstrap approximation of $Z$ (cf. Section \ref{sec4}). Corollary \ref{cor3} and Proposition \ref{prop1} suggest that, in general,  bootstrapping may not provide a more accurate approximation than the Gaussian approximation in terms of the dependence on dimension.
\end{remark}

\begin{remark}\label{rem:1}
Theorem \ref{BALL2} can be used to deduce Central Limit Theorems (CLTs) for $|W|^2$ under suitable conditions.
For example, if $\Sigma=I_d$, $\xi_i=X_i/\sqrt{n}$ for an i.i.d.\ sequence of random vectors $\{X_1,\dots, X_n\}$ with $\max_{1\leq j\leq d} \E(X_{ij}^4)\leq C$, then by Corollary \ref{cor3} and the CLT for chi-square random variables, we have, for $d\to \infty$ and $d=o(n)$,
\be{
\frac{|W|^2-d}{\sqrt{2d}}\rightarrow N(0,1) \quad \text{in distribution.}
}
This recovers Corollary 3 of \cite{PeSc18}, who proved the result by regarding $|W|^2$ as a quadratic function of $\{\xi_{ij}\}_{1\leq i\leq n, 1\leq j\leq d}$ and using the martingale CLT.

\cite[Corollary 1]{XuZhWu19} used Lindeberg's swapping argument to obtain an explicit error bound in approximating $|W|^2$ by $|Z|^2$. Their error bound yields an optimal result in approximating Pearson's chi-squared statistics. Under the setting of Corollary \ref{cor3}, their error bound also vanishes if $d=o(n)$.
In this regard, our main contribution is a new proof of such results using Stein's method (which works beyond sums of independent random vectors as demonstrated in the Appendix) and an application to the bootstrap approximation.
\end{remark}

\section{Application to bootstrap approximation on Balls}\label{sec4}

Let $X=\{X_i\}_{i=1}^n$ be a sequence of centered independent random vectors in $\mathbb{R}^d$ with finite fourth moments and consider the normalized sum $W:=n^{-1/2}\sum_{i=1}^nX_i$. 
Theoretical results developed in the previous section allows us to approximate the probability  $\IP(|W|\leq r)$ for $r\geq0$ by its Gaussian analog $\IP(|Z|\leq r)$, where $Z\sim N(0,\Sigma)$ and $\Sigma:=\Var(W)$ even when the dimension $d$ grows with the sample size $n$. 
Nevertheless, analytical evaluation of $\IP(|Z|\leq r)$ could be complicated for a general form of $\Sigma$ (and become impossible for unknown $\Sigma$) and thus we may still need an additional effort to resolve this issue for statistical application. 
This section develops bootstrap approximation for $\IP(|W|\leq r)$, one of the most popular methods to settle this sort of problem. Concrete applications are found in \cite{SpZh15}, \cite{Po15} and \cite{ChZh20}, for example.    
We remark that the approach works for bootstrap approximation on convex sets and non-centered balls, although we do not include the details in the paper.
\medskip

\subsection{Empirical bootstrap}

First we consider Efron's empirical bootstrap introduced by \cite{Ef79}. 
Let $X_1^*,\dots,X_n^*$ be i.i.d.~draws from the empirical distribution of $X$. That is, conditional on $X$, $X_1^*,\dots,X_n^*$ are independent and each $X_i^*$ is uniformly distributed on $\{X_1,\dots,X_n\}$. The bootstrap approximation of $W$ is then given by
\[
W^*:=\frac{1}{\sqrt n}\sum_{i=1}^n(X_i^*-\bar X),\qquad\text{where}\quad\bar X:=\frac{1}{n}\sum_{i=1}^nX_i.
\]
The following theorem provides a bootstrap analog of Theorem \ref{BALL2} and is used to give an approximation of $W$ by $W^*$ in Corollary \ref{coro-ef}.
\begin{theorem}\label{efron}
Set $\Sigma=\Sigma_W=\Var(W)$ and $Z\sim N(0,\Sigma)$. 
If $\Lambda_2(\Sigma)>0$, then
\ben{\label{ef-ga}
\E\left[\sup_{r\geq0}|\P(|W^*|\leq r\mid X)-\P(|Z|\leq r)|\right]
\leq C\Delta_n^*,
}
where
\bm{
\Delta_n^*:=\varkappa^{3/4}(\Sigma)\delta_1^{1/4}(W)+\varkappa^{2/3}(\Sigma)\delta_2^{1/3}(W)
+\varkappa^{1/2}(\Sigma)\left(\sum_{j=1}^d\sqrt{\frac{1}{n^2}\sum_{i=1}^n\E [X_{ij}^4]}\right)^{1/2}\\
+\varkappa^{2/3}(\Sigma)[\tr(\Sigma)]^{1/6}(\wh\delta+\|\Sigma\|_{op})^{1/6}\left(\frac{1}{n^2}\sum_{i=1}^n\E[| X_i|^4]\right)^{1/6}\\
+\varkappa^{3/4}(\Sigma)\left\{\frac{\wh\delta^{3/2}}{n^{3/2}}\sum_{i=1}^n\E[|X_i|^{3}]\right\}^{1/4}
+\varkappa^{2/3}(\Sigma)\left\{\frac{\wh\delta^{1/2}}{n^{3/2}}\sum_{i=1}^n\E[|X_i|^{3}]\right\}^{1/3}
}
with $\delta_1(W), \delta_2(W)$ as defined in Theorem \ref{BALL2}, $\wh\delta:=\E\|\wh\Sigma-\Sigma\|_{op}$ and $\wh{\Sigma}:=n^{-1}\sum_{i=1}^n(X_i-\bar{X})(X_i-\bar{X})^\top$.
\end{theorem}

\begin{corollary}\label{coro-ef}
Under the setting of Theorem \ref{efron}, we have
\ben{\label{eq-ef}
\sup_{r\geq0}|\P(|W|\leq r)-\P(|W^*|\leq r\mid X)|=O_p(\Delta_n^*)
}
as $n\to\infty$. 
Moreover, let $\alpha\in(0,1)$ and define 
\[
q_n^*(\alpha):=\inf\{x\in\IR:\P(|W^*|> x\mid X)\leq\alpha\}.
\]
Then we have $\P(|W|> q_n^*(\alpha))\to\alpha$ as $n\to\infty$, provided that $\Delta_n^*\to0$. 
\end{corollary}

\begin{remark}
If $\Sigma$ is invertible, it is possible to derive a bootstrap version of Theorem \ref{BALL} for balls not necessarily centered at 0, yielding an error bound that does not involve $\|\wh\Sigma-\Sigma\|_{op}$
and has the near optimal convergence rate with respect to the sample size $n$. 
\end{remark}

\label{rem:efron}
If 
\ben{\label{n40}
\Sigma=I_d,\quad \max_{1\leq i\leq n}\max_{1\leq j\leq d}\E|X_{ij}|^4\leq C,
} 
then we can bound $\Delta_n^*$ as (cf.~Corollary \ref{cor3} and its proof)
\ben{\label{star-bound}
\Delta_n^*\leq C\left\{\left(\frac{d}{n}\right)^{1/4}+\frac{1+\wh\delta^{3/8}}{n^{1/8}}
+(1+\wh\delta^{1/6})\left(\frac{d}{n}\right)^{1/6}+\left(\frac{d}{n}\right)^{1/3}\right\}.
}
Using \eqref{hs-bound0}--\eqref{hs-bound2}, we can bound $\wh\delta$ as
\ben{\label{naive-bound}
\wh\delta\leq \E\|\wh\Sigma-\Sigma\|_{H.S.}
\leq C\sqrt{\frac{1}{n^2}\sum_{i=1}^n\E[|X_i|^4]}
\leq C\frac{d}{\sqrt n}.
}
Thus $\Delta_n^*$ can be bounded by
\[
C\left\{\left(\frac{d}{n}\right)^{1/4}+\frac{1}{n^{1/8}}+\left(\frac{d^{6/5}}{n}\right)^{5/16}
+\left(\frac{d}{n}\right)^{1/6}+\left(\frac{d^{4/3}}{n}\right)^{1/4}+\left(\frac{d}{n}\right)^{1/3}\right\},
\]
which converges to 0 when $d^{4/3}/n\to0$. 

To get a better bound for $\wh\delta$, we need to impose an additional assumption on $X_i$. For example, this is the case when $X_i$ are sub-Gaussian random vectors. For a random variable $\xi$, we define its sub-Gaussian norm by
\[
\|\xi\|_{\psi_2}:=\inf\{t>0:\E\exp(\xi^2/t^2)\leq2\}.
\]
\begin{proposition}\label{prop:op-norm}
Suppose that there is a constant $L\geq1$ such that
\ben{\label{sub-gauss}
\|X_i\cdot u\|_{\psi_2}\leq L\sqrt{\E[|X_i\cdot u|^2]}\qquad\text{ for all }u\in\mathbb R^d\text{ and }i=1,\dots,n.
}
Then we have
\ben{\label{subgauss-bound}
\wh\delta\leq C\left(L^2\sqrt{\frac{\|\Sigma\|_{op}\tr(\Sigma)}{n}}+L^4\frac{\tr(\Sigma)}{n}\right).
}
\end{proposition}
For the i.i.d.~case, Proposition \ref{prop:op-norm} is essentially a special case of \cite[Theorem 4]{KoLo17}. The non-i.i.d.~extension can be obtained by a trivial modification of the proof of \cite[Theorem 9.2.4]{Ve18}; see the proof in Section \ref{sec5.8}. 
Consequently, if we additionally assume \eqref{sub-gauss}, then $\Delta_n^*\to0$ as $d/n\to0$ by \eqref{star-bound}. Examples satisfying \eqref{sub-gauss} are found in \cite[Section 3.4]{Ve18}. 
See \cite{ALPT11,SrVe13,Ti18} for alternative assumptions to get a better bound for $\wh\delta$.

\begin{remark}[Relation to \cite{Zh20}]\label{rem:zhilova}
From the above discussion, for centered balls, our error bound 
for the bootstrap approximation \eq{eq-ef} vanishes in probability when $d=o(n^{3/4})$ in the case \eq{n40}
and when $d=o(n)$ in the case \eq{n40} and \eq{sub-gauss}. For non-centered balls, it is possible to derive an error bound, based on Theorem \ref{BALL}, for the bootstrap approximation which vanishes in probability when $d=o(n^{1/2})$ in the case \eq{n40}.
Theorem 4.1 of \cite{Zh20} gives a non-asymptotic bound for  
$\sup_{A\in \mathcal{B}}|\P(W\in A)-\P(W^*\in A|X)|$.
Her bound vanishes when $d=o(n^{(K-2)/K})$, where $K\geq3$ is an integer determined by a distributional property of $X_i$ (cf.~Eq.(1.5) ibidem).
The distributional property is in general hard to verify and a sufficient condition was only given for the case $K=4$ (cf. Lemma 3.2 of \cite{Zh20}). 


\end{remark}

\subsection{Wild bootstrap}

Next we consider the wild (or multiplier) bootstrap, which was originally suggested in Section 7 of \cite{Wu86} (see also \cite{Li88}). Let $e_1,\dots,e_n$ be i.i.d.~variables independent of $X$ with $\E e_1=0$, $\E e_1^2=1$ and $\E e_1^4<\infty$. The wild bootstrap approximation of $W$ at the beginning of this section with multiplier variables $e_1,\dots,e_n$ is given by
\[
W^\circ:=\frac{1}{\sqrt n}\sum_{i=1}^ne_iX_i.
\]
In this setting, we can establish the following wild bootstrap version of Theorem \ref{efron},
which is then used to give an approximation of $W$ by $W^\circ$. 
\begin{theorem}\label{wild}
Under the setting of Theorem \ref{efron}, we have
\be{
\E\left[\sup_{r\geq0}|\P(|W^\circ|\leq r\mid X)-\P(|Z|\leq r)|\right]
\leq C_e\Delta_n^\circ,
}
where $C_e>0$ is a constant depending only on $\E [e_1^4]$, and $\Delta_n^\circ$ is defined as $\Delta_n^*$ in Theorem \ref{efron} with $\wh\Sigma$ replaced by $\bar\Sigma:=n^{-1}\sum_{i=1}^nX_iX_i^\top$. 
\end{theorem}

\begin{corollary}\label{coro-w}
Under the setting of Theorem \ref{wild},
if the law of $e_1$ does not depend on $n$, then we have
\ben{\label{eq-w}
\sup_{r\geq0}|\P(|W|\leq r)-\P(|W^\circ|\leq r\mid X)|=O_p(\Delta_n^\circ)
}
as $n\to\infty$. 
Moreover, let $\alpha\in(0,1)$ and define 
\[
q_n^\circ(\alpha):=\inf\{x\in\IR:\P(|W^\circ|> x\mid X)\leq\alpha\}.
\]
Then we have $\P(|W|> q_n^\circ(\alpha))\to\alpha$ as $n\to\infty$, provided that $\Delta^\circ_n\to0$. 
\end{corollary}

\begin{remark}
Again, it is possible to derive a wild bootstrap version of Theorem \ref{BALL} when $\Sigma$ is invertible. 
Also, $\Delta^\circ_n$ can be bounded as in \eqref{star-bound} with $\wh\delta$ replaced by $\E\|\bar\Sigma-\Sigma\|_{op}$. Since both \eqref{naive-bound} and \eqref{subgauss-bound} continue to hold true while replacing $\wh\delta$ by $\E\|\bar\Sigma-\Sigma\|_{op}$, similar comments in Remark \ref{rem:zhilova} are applicable to $\Delta_n^\circ$. 
\end{remark}

\begin{remark}[Relation to \cite{Zh20} cont.]
Theorem 4.3 of \cite{Zh20} establishes a non-asymptotic bound for  
$\sup_{A\in \mathcal{B}}|\P(W\in A)-\P(W^\circ\in A|X)|$.
Her bound vanishes for $d=o(n^{1/2})$ assuming $\E[e_1^3]=1$ and certain distributional properties of $X_i$. 
As mentioned in Remark \ref{rem:zhilova}, it is possible to derive an error bound based on our Theorem \ref{BALL} allowing the same growth rate of $d$ but without these two assumptions. 
Moreover, the dimensional dependence can be further improved when considering centered balls. 

\end{remark}

\section{Proofs}\label{sec5}

We first introduce some notation used throughout the proofs. 
For two vectors $x,y\in\mathbb{R}^d$, $x\cdot y$ denotes their inner product. 
For two $d\times d$ matrices $M$ and $N$, we write $\langle M,N\rangle_{H.S.}$ for their Hilbert-Schmidt inner product. 

For real-valued functions on $\IR^d$ we will write $\partial_i f(x)$ for $\partial f(x)/\partial x_i$, $\partial_{ij} f(x)$ for $\partial^2 f(x)/(\partial x_i\partial x_j)$ and so forth. 
We write $\nabla f$ and $\Hess f$ for the gradient and Hessian matrix of $f$, respectively. 
In addition, following \cite{Ra19a,Ra19b}, we denote by $\nabla^rf(x)$ the $r$-th derivative of $f$ at $x$ regarded as an $r$-linear form: The value of $\nabla^rf(x)$ evaluated at $u_1,\dots,u_r\in\mathbb{R}^d$ is given by
\[
\langle \nabla^r f(x),u_1\otimes\cdots\otimes u_r\rangle=\sum_{j_1,\dots,j_r=1}^d\partial_{j_1,\dots,j_r}f(x)u_{1,j_1}\cdots u_{r,j_r}.
\]
When $u_1=\cdots=u_r=:u$, we write $u_1\otimes\cdots\otimes u_r=u^{\otimes r}$ for short. 

For any $r$-linear form $T$, its injective norm is defined by
\[
|T|_\vee:=\sup_{|u_1|\vee\cdots\vee|u_r|\leq1}|\langle T,u_1\otimes\cdots\otimes u_r\rangle|.
\] 
For an $(r-1)$-times differentiable function $h:\mathbb{R}^d\to\mathbb{R}$, we write
\[
M_r(h):=\sup_{x\neq y}\frac{|\nabla^{r-1}h(x)-\nabla^{r-1}h(y)|_\vee}{|x-y|}.
\]
Note that $M_r(h)=\sup_{x\in\mathbb{R}^d}|\nabla^{r}h(x)|_\vee$ if $h$ is $r$-times differentiable. We refer to the beginning of \cite[Section 2]{Ra19a} and \cite[Section 5]{Ra19b} for more details about these notation. 

Finally, we refer to the following bound for derivatives of the $d$-dimensional standard normal density $\phi$, which will be used several times in the following (cf.~the inequality after Eq.(4.9) of \cite{Ra19b}):
\ben{\label{raic4.9}
\int_{\IR^d} |\langle \nabla^s \phi(z), u^{\otimes s}\rangle | dz \leq C_s |u|^s\quad\text{for any fixed integer $s$},
} 
where $C_s$ is a constant depending only on $s$.

\subsection{Basic decomposition}

The proofs for Theorems \ref{convex} and \ref{BALL}--\ref{BALL2} start with approximating the indicator function $1_A$ for $A\in\mcl{A}$ or $A\in\mcl{B}$ by an appropriate smooth function $h$. Then, the problem amounts to establishing an appropriate bound for $\E h(W)-\E h(Z)$. 
To accomplish this, in the proofs of Theorems \ref{convex} and \ref{BALL}, we will make use of a decomposition of $\E h(W)-\E h(Z)$ derived from the exchangeable pair approach in Stein's method for multivariate normal approximation by \cite{ChMe08} and \cite{ReRo09} along with a symmetry argument by \cite{FaKo20a,FaKo20b} (cf.~\eq{n08}--\eq{n09} below).

Given a twice differentiable function $h:\IR^d\to\IR$ with bounded partial derivatives, we consider the Stein equation
\ben{   \label{n03}
  \langle\Hess f(w),\Sigma\rangle_{H.S.}-w\cdot \nabla f(w)=h(w)-\E h(Z),\qquad w\in \IR^d.
}
It can be verified directly
that
\ben{\label{sol}
f(w)=\int_0^1-\frac{1}{2(1-s)} \int_{\IR^d}
    \bkle{h(\sqrt{1-s}w+\sqrt{s}\Sigma^{1/2}z)-\E h (Z)}\phi(z)dzds
}
is a solution to \eq{n03} (cf.~\cite{Go91} and \cite{Me09}). 
In the following we assume that $f$ is thrice differentiable with bounded partial derivatives. This is true if $\Sigma$ is invertible or $h$ is thrice differentiable with bounded partial derivatives. 

Let $\{\xi_1',\dots, \xi_n'\}$ be an independent copy of $\{\xi_1,\dots, \xi_n\}$, and let $I$ be a random index uniformly chosen from $\{1,\dots, n\}$ and independent of $\{\xi_1,\dots, \xi_n, \xi'_1,\dots, \xi'_n\}$.
Define $W'=W-\xi_I+\xi_I'$. It is easy to verify that $(W, W')$ has the same distribution as $(W', W)$ (exchangeability) and 
\ben{\label{n02}
\E(W'-W\mid W)=-\frac{W}{n}.
}
From exchangeability and \eq{n02}, we have, with $D=W'-W$,
\begin{align}
0&=\frac{n}{2}\E[ D\cdot(\nabla f (W')+\nabla f(W))]\nonumber\\
&=\E\left[\frac{n}{2}D\cdot(\nabla f(W')-\nabla f(W))+nD\cdot\nabla f(W)\right]\nonumber\\
&=\E\left[\frac{n}{2}\sum_{j,k=1}^d D_jD_k\partial_{jk}f (W)+R_2+ n D\cdot\nabla f(W)\right]\label{n04}\\
&=\E\left[\langle\text{Hess}f(W),\Sigma\rangle_{H.S.}-R_1+R_2- W\cdot \nabla f(W)\right],\nonumber
\end{align}
where
\ban{
R_1&=\sum_{j,k=1}^d \E\bklg{(\Sigma_{jk}- \frac{n}{2}D_{j}D_{k})\partial_{jk} f(W)},\\
R_2&=\frac{n}{2}\sum_{j,k,l=1}^d \E \bklg{D_jD_kD_lU\partial_{jkl}f(W+(1-U)D) }\label{n09-0}
}
and $U$ is a uniform random variable on $[0,1]$ independent of everything else. 
From \eq{n03} and \eq{n04}, we have
\ben{\label{n18}
\E h(W)-\E h(Z)=R_1-R_2.
}
We further rewrite $R_1$ and $R_2$ respectively as follows. 
First, set
\ba{
V=(V_{jk})_{1\leq j,k\leq d}&:=\left(\E\left[\Sigma_{jk}- \frac{n}{2}D_{j}D_{k}\mid\xi\right]\right)_{1\leq j,k\leq d}.
}
Then we evidently have
\ben{\label{n25}
R_1=\sum_{j,k=1}^d\E [V_{jk}\partial_{jk}f(W)]=\E\langle V,\Hess f(W)\rangle_{H.S.}. 
}
Also, one can easily verify that (cf.~Eq.(22) of \cite{CCK14})
\ben{\label{v-eq}
V=\Sigma-\frac{1}{2}\sum_{i=1}^n \E[\xi_i \xi_i^\top]-\frac{1}{2} \sum_{i=1}^n \xi_i \xi_i^\top=(\Sigma-\Var(W))-\frac{1}{2}\sum_{i=1}^n(\xi_i\xi_i^\top-\E[\xi_i\xi_i^\top]).
}
Next, by exchangeability we have
\besn{\label{n08}
&\E[D_jD_kD_lU\partial_{jkl}f(W+(1-U)D)]\\
&=-\E[D_jD_kD_lU\partial_{jkl}f(W'-(1-U)D)]\\
&=-\E[D_jD_kD_lU\partial_{jkl}f(W+UD)].
}
Hence we obtain
\begin{align}
R_2&=\frac{n}{4}\sum_{j,k,l=1}^d\E[D_jD_kD_lU\{\partial_{jkl}f(W+(1-U)D)-\partial_{jkl}f(W+UD)\}].
\label{n09}
\end{align}

\subsection{Proof of Theorem \ref{convex}}

Since $\Sigma^{-1/2}W=\sum_{i=1}^n\Sigma^{-1/2}\xi_i$ and $\{\Sigma^{-1/2}x:x\in A\}\in\mcl{A}$ for all $A\in\mcl{A}$, it suffices to consider the case $\Sigma=I_d$. 
The proof is a combination of \cite{Be03}'s smoothing, the decomposition \eqref{n18}, and a recursive argument by \cite{Ra19a}. 

Fix $\beta_0>0$. Define
\ben{\label{n11}
K(\beta_0)=\sup_{W} \frac{\sup_{A\in \mathcal{A}}|\P(W\in A)-\P(Z\in A)|}{\max \left\{\beta_0, \Psi(\delta_\mathcal{A}(W, I_d)) \right\}},
}
where the first supremum is taken over the family of all sums $W=\sum_{i=1}^n \xi_i$ of $n$ independent mean-zero random vectors with $\E |\xi_i|^4<\infty$. 
We will obtain a recursive inequality for $K(\beta_0)$ and prove that 
\ben{\label{n22}
K(\beta_0)\leq  Cd^{1/4} 
}
for an absolute constant $C$ that does not depend on $\beta_0$. Eq.\eq{n01} then follows by sending $\beta_0\to 0$.

Now we fix a $W=\sum_{i=1}^n \xi_i$, $n\geq 1$, in the aforementioned family (will take sup in \eq{n27}). Let
\ben{\label{n12}
\bar \beta=\max \left\{\beta_0, \Psi(\delta_\mathcal{A}(W, I_d)) \right\}.
}
Next, recall that
$\mathcal{A}$ is the collection of all convex sets in $\IR^d$. For
$A\in\mathcal{A}$, $\eps>0$, define 
\be{
  A^{\eps} = \{x\in \mathbb{R}^d\,:\, dist(x,A)\leq \eps\},
}
where $dist(x, A)=\inf_{y\in A}|x-y|$.


\begin{lemma}[Lemma 2.3 of \cite{Be03}]\label{nl1}
For any $A\in \mathcal{A}$ and $\eps>0$, there exists a function $h_{A,\eps}$ (which depends only on $A$ and $\eps$) such that
\be{
h_{A,\eps}(x)=1\ \text{for}\ x\in A,\qquad h_{A,\eps}(x)=0\ \text{for}\ x\in \IR^d\backslash A^\eps,\qquad 0\leq h_{A,\eps}(x)\leq 1,
}
and
\ben{\label{n06}
M_1(h_{A,\eps})\leq \frac{C}{\eps}, \qquad M_2(h_{A,\eps})\leq \frac{C}{\eps^2},
}
where $C$ is an absolute constant that does not depend on $A$ and $\eps$.
\end{lemma}

\begin{lemma}[Theorem 4 of \cite{Ba93}]\label{nl2}
Let $\phi$ be the standard Gaussian density on $\IR^d, d\geq 2$, and let $A$ be a convex set in $\IR^d$. Then
\ben{\label{n28}
\int_{\partial A} \phi \leq 4 d^{1/4}.
}
\end{lemma}
From the Gaussian anti-concentration inequality \eq{n28}, it is not difficult to obtain the following smoothing lemma.

\begin{lemma}[Lemma 4.2 of \cite{FaRo15}]\label{lem4}
For any $d$-dimensional random vector $W$ and any $\eps>0$,
\ben{\label{n17}
  \sup_{A\in \mathcal{A}} |P(W\in A)-P(Z\in A)| \leq 4d^{1/4}\eps +
    \sup_{A\in\mathcal{A}}\babs{\E h_{A,\eps}(W)-\E h_{A,\eps}(Z)},
}
where $h_{A,\eps}$ is as in Lemma \ref{nl1}.
\end{lemma}

The following lemma can be shown by elementary calculation, so we omit its proof.
\begin{lemma}\label{lem:psi}
$\Psi$ is an increasing function on $(0,\infty)$. Moreover, $\Psi(cx)\leq(c+\Psi(c))\Psi(x)$ for all $x>0$ and $c\geq1$.
\end{lemma}

We now fix $A\in \mathcal{A}$ (will take sup in \eq{n26}), $0<\eps\leq 1$, write $h:=h_{A, \eps}$ and proceed to bound $|\E h(W)-\E h(Z)|$ by the decomposition \eq{n18}. 
Consider the solution $f$ to the Stein equation \eqref{n03} with $\Sigma=I_d$, which is given by \eqref{sol}. 
Since $h$ has bounded partial derivatives up to the second order and $\Sigma=I_d$ is invertible, $f$ is thrice differentiable with bounded partial derivatives. Using the integration by parts formula, we have for $1\leq j, k, l\leq d$ and any constant $0\leq c_0\leq 1$ that
\besn{   \label{n05}
  \partial_{jk} f (w) & =
  \int_0^{c_0}  \frac{1}{2\sqrt{s}} \int_{\IR^d}
     \partial_j h(\sqrt{1-s}w+\sqrt{s}z) \partial_k \phi (z) dz ds\\
   &\quad +\int_{c_0}^1- \frac{1}{2s}
        \int_{\IR^d} h (\sqrt{1-s}w+\sqrt{s}z) \partial_{jk}\phi (z) dz ds
}
and
\besn{\label{n10}
  \partial_{jkl} f (w) & =
  \int_0^{c_0} \frac{\sqrt{1-s}}{2\sqrt{s}} \int_{\IR^d}
     \partial_{jk} h (\sqrt{1-s}w+\sqrt{s}z) \partial_l \phi (z) dzds\\
     &\quad +\int_{c_0}^1 -\frac{\sqrt{1-s}}{2s}
        \int_{\IR^d} \partial_jh (\sqrt{1-s}w+\sqrt{s}z) \partial_{kl} \phi (z) dz ds.
}

We first bound $R_1$ in \eqref{n25}. We will utilize the following lemma. 
\begin{lemma}[Lemma 4.3 of \cite{FaRo15}]\label{lem5}
For $k\geq 1$ and each map $a: \{1,\dots,d\}^k\rightarrow \IR$, we have
\ben{\label{n07}
\int_{\mr}\left( \sum_{i_1,\dots, i_k=1}^d a(i_1,\dots, i_k)\frac{\partial_{i_1\dots i_k} \phi(z)}{\phi(z)}  \right)^2\phi(z)dz 
\leq k! \sum_{i_1,\dots, i_k=1}^d \left( a(i_1,\dots, i_k)  \right)^2.
}
\end{lemma}
Now, using the expression of $\partial_{jk} f$ in \eq{n05} with $c_0=\eps^2$, we have
\be{
R_1=R_{11}+R_{12},
}
where
\be{
R_{11}= \sum_{j,k=1}^d\E\left[ V_{jk} \int_0^{\eps^2} \frac{1}{2\sqrt{s}} \int_{\IR^d}\partial_j h(\sqrt{1-s}W+\sqrt{s} z)\partial_k \phi (z)dzds\right]
}
and 
\be{
R_{12}=\sum_{j,k=1}^d\E\left[ V_{jk}\int_{\eps^2}^1 -\frac{1}{2s} \int_{\IR^d} h(\sqrt{1-s}W+\sqrt{s} z)\partial_{jk} \phi (z)dzds\right].
}
For $R_{11}$, we use the Cauchy-Schwarz inequality and the bounds \eq{n06} and \eq{n07}, and obtain
\besn{\label{r11}
\left| R_{11} \right|
&=\left| \int_0^{\eps^2} \frac{1}{2\sqrt{s}} \int_{\IR^d} \E \left[ \sum_{j=1}^d\partial_j h(\sqrt{1-s}W+\sqrt{s} z)\sum_{k=1}^d V_{jk}\frac{\partial_k \phi(z)}{\phi(z)} \right]\phi(z)dz       ds    \right|  \\
&\leq\frac{C}{\eps} \int_0^{\eps^2} \frac{1}{2\sqrt{s}} \int_{\IR^d} \E \left[ \left\{\sum_{j=1}^d\left(\sum_{k=1}^d V_{jk}\frac{\partial_k \phi(z)}{\phi(z)}\right)^2\right\}^{1/2} \right]\phi(z)dz       ds      \\
&\leq\frac{C}{\eps} \int_0^{\eps^2} \frac{1}{2\sqrt{s}} \left\{\int_{\IR^d} \E \left[ \sum_{j=1}^d\left(\sum_{k=1}^d V_{jk}\frac{\partial_k \phi(z)}{\phi(z)}\right)^2 \right]\phi(z)dz \right\}^{1/2}      ds      \\
&\leq\frac{C}{\eps} \int_0^{\eps^2} \frac{1}{2\sqrt{s}} \left\{\E[\sum_{j=1}^d\sum_{k=1}^d V_{jk}^2] \right\}^{1/2}      ds      
\leq C\left\{ \sum_{j,k=1}^d \E [V_{jk}^2] \right\}^{1/2}.
}
The triangle inequality yields, for $V$ in \eq{v-eq},
\begin{align*}
\left\{ \sum_{j,k=1}^d \E [V_{jk}^2] \right\}^{1/2}
\leq\|I_d-\Var(W)\|_{H.S.}
+\frac{1}{2}\left\{ \sum_{j,k=1}^d\Var \left[\sum_{i=1}^n\xi_{ij}\xi_{ik}\right]\right\}^{1/2}.
\end{align*}
Moreover,
\ben{\label{v-est}
\Var \left[\sum_{i=1}^n\xi_{ij}\xi_{ik}\right]  
=\sum_{i=1}^n\Var \left[\xi_{ij}\xi_{ik}\right]  
\leq \sum_{i=1}^n\E \left[\xi_{ij}^2\xi_{ik}^2\right], 
}
Therefore,
we obtain
\ba{
\left| R_{11} \right| 
\leq C\delta_\mathcal{A}(W,I_d).
}
Applying similar arguments, we have, for $R_{12}$,
\besn{\label{r12}
|R_{12}|
&=\left| \int_{\eps^2}^1 (-\frac{1}{2s})\int_{\IR^d} \E    \left[ h(\sqrt{1-s}W+\sqrt{s} z) \sum_{j,k=1}^dV_{jk}\frac{\partial_{jk} \phi(z)}{\phi(z)} \right]\phi(z)dz     ds \right|    \\
&\leq \int_{\eps^2}^1\frac{1}{2s}\int_{\IR^d} \E \left|\sum_{j,k=1}^dV_{jk}\frac{\partial_{jk} \phi(z)}{\phi(z)}\right|\phi(z)dz     ds\\
&\leq \int_{\eps^2}^1\frac{1}{2s}\left\{\int_{\IR^d} \E \left[ \left(\sum_{j,k=1}^dV_{jk}\frac{\partial_{jk} \phi(z)}{\phi(z)}\right)^2 \right] \phi(z)dz\right\}^{1/2}     ds    \\
&\leq C\int_{\eps^2}^1\frac{1}{2s}\left\{\E[\sum_{j,k=1}^dV_{jk}^2]\right\}^{1/2}     ds    
\leq C|\log \eps| \delta_\mathcal{A}(W,I_d).
}
Therefore,
\ben{\label{n19}
|R_1|\leq C (|\log \eps|\vee 1) \delta_\mathcal{A}(W,I_d).
}
Next, we bound $R_2$. 
Take $0<\eta\leq1$ arbitrarily.
Using the expression of $\partial_{jkl} f$ in \eq{n10} with $c_0=\eta^2$ and the two equivalent expressions \eqref{n09-0} and \eqref{n09} for $R_2$, we have
\be{
R_2=R_{21}+R_{22},
}
where
\bes{
R_{21}=&\frac{1}{2} \sum_{i=1}^n \sum_{j,k,l=1}^d \E \Big\{ U(\xi_{ij}'-\xi_{ij})(\xi_{ik}'-\xi_{ik})(\xi_{il}'-\xi_{il})
\int_0^{\eta^2} \frac{\sqrt{1-s}}{2\sqrt{s}} \\
& \times \int_{\IR^d} \partial_{jk} h(\sqrt{1-s}(W+(1-U)(\xi_i'-\xi_i))+\sqrt{s}z)\partial_l \phi(z) dz ds \Big\}
}
and
\besn{\label{n32}
R_{22}=&\frac{1}{4} \sum_{i=1}^n \sum_{j,k,l=1}^d \E \Big\{ U(\xi_{ij}'-\xi_{ij})(\xi_{ik}'-\xi_{ik})(\xi_{il}'-\xi_{il})
\int_{\eta^2}^1 -\frac{\sqrt{1-s}}{2s} \\
& \times \int_{\IR^d} [\partial_jh(\sqrt{1-s}(W+(1-U)(\xi_i'-\xi_i))+\sqrt{s}z)-\partial_jh(\sqrt{1-s}(W+U(\xi_i'-\xi_i))+\sqrt{s}z)] \partial_{kl} \phi(z) dz ds \Big\} \\
=& \frac{1}{4} \sum_{i=1}^n \sum_{j,k,l, m=1}^d \E \Big\{ U (1-2U)(\xi_{ij}'-\xi_{ij})(\xi_{ik}'-\xi_{ik})(\xi_{il}'-\xi_{il})(\xi_{im}'-\xi_{im})
\int_{\eta^2}^1 -\frac{1-s}{2s} \\
&\times \int_{\IR^d} \partial_{jm} h(\sqrt{1-s}(W+(U+(1-2U)U')(\xi_i'-\xi_i))+\sqrt{s}z) \partial_{kl} \phi(z) dz ds \Big\},
}
where $U'$ is a uniform random variable on $[0,1]$ independent of everything else and we used the mean value theorem in the last equality. 
Let $W^{(i)}=W-\xi_i$ for $i\in \{1,\dots, n\}$. We will use the fact that $\nabla h$ is non-zero only in $A^\eps\backslash A$ and bound
\be{
\P(\sqrt{1-s} W^{(i)}\in A_i^\eps\backslash A_i \mid U, U', \xi_i, \xi_i'),
}
where $0<s<1$ and $A_i$ is a convex set which may depend on $U$, $U'$, $\xi_i$, $\xi_i'$, $s$ and $z$.
We have
\besn{\label{n14}
&\P(\sqrt{1-s} W^{(i)}\in A_i^\eps\backslash A_i \mid U, U', \xi_i, \xi_i')\\
=&\P( W^{(i)} \in  \frac{1}{\sqrt{1-s}} (A_i^\eps\backslash A_i) \mid U, U', \xi_i, \xi_i')\\
\leq & 4d^{1/4} \frac{\eps}{ \sqrt{1-s}} + 2\sup_{A\in \mathcal{A}} \left|\P( W^{(i)} \in A) -P(Z\in A)    \right|,
}
where we used the $4d^{1/4}$ upper bound for the Gaussian surface area of any convex set in Lemma \ref{nl2}. From \eq{n11}, and regarding $W^{(i)}=\sum_{j: j\ne i} \xi_j+0$ as a sum of $n$ independent centered random vectors, we have
\besn{\label{n13}
&\sup_{A\in \mathcal{A}} \left|\P(W^{(i)} \in A) -P(Z\in A)    \right|\\
\leq & K(\beta_0) \max \left\{ \beta_0,  \Psi(\delta_\mathcal{A}(W^{(i)}, I_d))  \right\}.
}
Since
\ba{
\|\Var(W)-\Var(W^{(i)})\|_{H.S.}
&=\sqrt{\sum_{j,k=1}^d(\E\xi_{ij}\xi_{ik})^2}
\leq \sqrt{\E|\xi_i|^4}
}
and $\sqrt{x}+\sqrt{y}\leq\sqrt{2(x+y)}$ for any $x,y\geq0$, we have
\ba{
\delta_\mathcal{A}(W^{(i)},\Sigma)
&\leq \|I_d-\Var(W)\|_{H.S.}
+\sqrt{\E[|\xi_i|^4]}
+\sqrt{\sum_{j=1\atop j\ne i}^n \E[|\xi_j|^4]}\\
&\leq \|I_d-\Var(W)\|_{H.S.}+\sqrt{2\sum_{j=1}^n \E[|\xi_j|^4]} 
\leq\sqrt{2}\delta_\mathcal{A}(W, I_d).
} 
Hence, we obtain by Lemma \ref{lem:psi} 
\ba{
\Psi(\delta_\mathcal{A}(W^{(i)},\Sigma))\leq C\Psi(\delta_\mathcal{A}(W, I_d))\leq C\bar\beta.
}
Thus we conclude
\ben{\label{n15}
\sup_{A\in \mathcal{A}} \left|\P(W^{(i)} \in A) -P(Z\in A)    \right|\leq K(\beta_0)\max \left\{ \beta_0,  C\bar\beta\right\}
\leq C K(\beta_0)\bar\beta.
}
Applying \eq{n06}, \eq{n14}, \eq{n15}, and \eqref{raic4.9}, we have
\ben{\label{n16-0}
|R_{21}|
\leq \frac{C}{\eps^2} \sum_{i=1}^n \E[|\xi_i|^3] \left( d^{1/4} \eps +K(\beta_0)     \bar \beta  \right)\eta
}
and
\ben{\label{n16}
|R_{22}|\leq \frac{C}{\eps^2} \sum_{i=1}^n \E[|\xi_i|^4 ] \left( d^{1/4} \eps +K(\beta_0)     \bar \beta  \right) |\log\eta|.
}
We first consider \textbf{Case 1}: $\bar{\beta}\leq 0.5/d^{1/4}$ (cf. \eq{n12}).
Now, if $\sum_{i=1}^n \E[|\xi_i|^4]<\sum_{i=1}^n \E[|\xi_i|^3]$, choose $\eta=\sum_{i=1}^n \E[|\xi_i|^4]/\sum_{i=1}^n \E[|\xi_i|^3]<1$. Note that we have by the Cauchy-Schwarz inequality
\[
\sum_{i=1}^n \E[|\xi_i|^3]\leq\sqrt{\sum_{i=1}^n \E[|\xi_i|^2]\sum_{i=1}^n \E[|\xi_i|^4]}
=\sqrt{d\sum_{i=1}^n \E[|\xi_i|^4]}.
\]
Thus we obtain
\[
|\log \eta|\leq\frac{1}{2}\log d-\frac{1}{2}\log\left(\sum_{i=1}^n \E[|\xi_i|^4]\right).
\]
Since $(\sum_{i=1}^n \E [|\xi_i|^4])^{1/2}\leq \bar \beta$ and in the case under consideration, $\bar \beta\leq 0.5/d^{1/4}$, we have
\be{
|\log \eta|\leq C |\log (\sum_{i=1}^n \E [|\xi_i|^4])|.
} 
Therefore, \eqref{n16-0}--\eqref{n16} and $\Psi$ is an increasing function from Lemma \ref{lem:psi} yield
\besn{\label{n16-2}
|R_{21}|+|R_{22}|&\leq  \frac{C}{\eps^2} \sum_{i=1}^n \E[|\xi_i|^4] \left( d^{1/4} \eps +K(\beta_0)     \bar \beta  \right)\left(\left|\log\left(\sum_{i=1}^n \E[|\xi_i|^4]\right)\right|\vee 1\right)\\
&\leq  \frac{C\Psi^2(\delta_\mathcal{A}(W, I_d))}{\eps^2}\left( d^{1/4} \eps +K(\beta_0)     \bar \beta  \right).
}
This inequality also holds true if $\sum_{i=1}^n \E[|\xi_i|^4]\geq\sum_{i=1}^n \E[|\xi_i|^3]$ by taking $\eta=1$ in \eqref{n16-0}--\eqref{n16}. 
From \eq{n17}, \eq{n18}, \eq{n19}, \eq{n16-2}, we have
\besn{\label{n26}
&\sup_{A\in \mathcal{A}} |\P(W\in A)-\P(Z\in A)|\\
 \leq & 4d^{1/4}\eps+C (|\log \eps|\vee 1)\delta_\mathcal{A}(W, I_d) +\frac{C\Psi^2(\delta_\mathcal{A}(W, I_d))}{\eps^2}  \left( d^{1/4} \eps +K(\beta_0)     \bar \beta  \right) .
}
Choose $\eps=\min\{ \sqrt{2C} \Psi(\delta_\mathcal{A}(W, I_d)) ,1\}$ with the same absolute constant $C$ as in the third term on the right-hand side of \eq{n26}.
If $\eps<1$, then from \eq{n26} and $\bar{\beta}\leq 0.5/d^{1/4}$ in Case 1, we have 
$|\log \eps|\leq C|\log(\delta_{\mathcal{A}}(W, I_d))|$ and
\be{
\sup_{A\in \mathcal{A}} |\P(W\in A)-\P(Z\in A)|\leq \left( Cd^{1/4} +\frac{K(\beta_0)}{2} \right) \bar \beta;
}
hence
\ben{\label{n20}
\frac{\sup_{A\in \mathcal{A}} |\P(W\in A)-\P(Z\in A)|}{\bar \beta}\leq Cd^{1/4} +\frac{K(\beta_0)}{2}.
}
If $\eps=1$, then $\sum_{i=1}^n \E [|\xi_i|^4]$ and $\bar \beta$ are bounded away from 0 by an absolute constant; hence
\ben{\label{n29}
\frac{\sup_{A\in \mathcal{A}} |\P(W\in A)-\P(Z\in A)|}{\bar \beta}\leq \frac{1}{\bar \beta}\leq C.
}

We now consider \textbf{Case 2}: $\bar \beta>0.5/d^{1/4}$. We trivially estimate
\ben{\label{n21}
\frac{\sup_{A\in \mathcal{A}} |\P(W\in A)-\P(Z\in A)|}{\bar \beta}\leq \frac{1}{\bar \beta}\leq  C d^{1/4}.
}
Combining \eq{n20}, \eq{n29} and \eq{n21}, we obtain
\be{
 \frac{\sup_{A\in \mathcal{A}} |\P(W\in A)-\P(Z\in A)|}{\bar \beta}
\leq C d^{1/4}+\frac{K(\beta_0)}{2}.
}
Note that the right-hand side of the above bound does not depend on $W$. Taking supremum over $W$, we obtain
\ben{\label{n27}
K(\beta_0)\leq C d^{1/4}+\frac{K(\beta_0)}{2}.
}
This implies $\eq{n22}$, hence \eq{n01}.

\subsection{Proof of Proposition \ref{prop2}}

It is not difficult to see that \cite{Na76}'s example indeed satisfies the conditions stated in the proposition. We briefly summarize the construction for the sake of completeness. 

First, given an integer $n\geq3$, let $\{\eta_i\}_{i=1}^n$ be i.i.d.~variables such that
\[
\IP(\eta_1< y)=\Phi\left(\frac{y+a_n}{\sigma_n}\right)(1-p_n)+p_n1_{(x_n,\infty)}(y),\qquad y\in\mathbb{R},
\]
where $\Phi$ is the standard normal distribution function and $x_n,p_n,a_n,\sigma_n$ are positive constants satisfying the following conditions:
\ba{
x_n&=\frac{\sqrt{n}}{\log n},&
p_nx_n&=a_n(1-p_n),&
p_nx_n^2&=\frac{1}{2},&
\frac{1}{2}+(\sigma_n^2+a_n^2)(1-p_n)&=1.
}
By construction, we have 
\ba{
\E\eta_1&=0,&\E[\eta_1^2]&=1,&\E[\eta_1^3]&=x_n-(3a_n\sigma_n^2+a_n^3)(1-p_n)
} 
and
\[
\frac{x_n^2}{2}\leq\E[\eta_1^4]\leq\frac{x_n^2}{2}+3(a_n+\sigma_n)^4.
\]
Moreover, \cite{Na76} has shown that, for sufficiently large $n$, 
\[
\IP\left(\frac{1}{\sqrt n}\sum_{i=1}^n\eta_i<0\right)-\frac{1}{2}>\frac{\E[\eta_1^3]}{7\sqrt{2\pi n}}.
\]
Since we have
\ba{
\frac{x_n}{\sqrt 2}\leq\sqrt{\E[\eta_1^4]}\leq\frac{x_n}{\sqrt 2}+\sqrt{3}(a_n+\sigma_n)^2
\leq\frac{\E[\eta_1^3]}{\sqrt 2}+\frac{3a_n\sigma_n^2+a_n^3}{\sqrt 2}+\sqrt{3}(a_n+\sigma_n)^2,
}
and $\sigma_n^2+a_n^2=\frac{1}{2(1-p_n)}\leq1$ for sufficiently large $n$, we conclude
\[
\IP\left(\frac{1}{\sqrt n}\sum_{i=1}^n\eta_i<0\right)-\frac{1}{2}>\frac{\sqrt{\E[\eta_1^4]}}{8\sqrt{\pi n}}
\]
for sufficiently large $n$. 

Now let $\zeta_{i,j}$ ($i,j=1,2,\dots)$ be independent standard normal variables independent of $\{\xi_i\}_{i=1}^n$. Then we define the independent random vectors $\{\xi_i\}_{i=1}^n$ in $\mathbb{R}^d$ by 
\[
\xi_i:=\frac{1}{\sqrt n}(\eta_i,\zeta_{i,1},\dots,\zeta_{i,d-1})^\top,\qquad i=1,\dots,n.
\] 
We have
\ba{
\sum_{i=1}^n\E|\xi_i|^4
=n\E[|\xi_1|^4]\leq\frac{1}{n}\left\{2\E[\eta_1^4]+2\E \Big[\big(\sum_{j=1}^{d-1}\zeta_{i,j}^2\big)^2 \Big]\right\}
\leq\frac{2\E[\eta_1^4]+6d^2}{n}.
}
Therefore, if $d\leq \sqrt{n}/\log n=x_n$, we obtain
\ba{
\sum_{i=1}^n\E[|\xi_i|^4]
\leq\frac{14\E[\eta_1^4]}{n}.
}
Thus, we conclude
\[
\IP\left(\frac{1}{\sqrt n}\sum_{i=1}^n\eta_i<0\right)-\frac{1}{2}>C_0\sqrt{\sum_{i=1}^n\E[|\xi_i|^4]},
\]
where $C_0:=(8\sqrt{14\pi})^{-1}$. Hence, for $A=\{x\in\mathbb{R}^d:x_1=0\}$, we have
\[
|\IP(W\in A)-\IP(Z\in A)|>C_0\sqrt{\sum_{i=1}^n\E[|\xi_i|^4]}.
\]
This completes the proof.

\subsection{Proof of Theorem \ref{BALL}}

We first note that, for any $d\times d$ orthogonal matrix $U$, we have $UW=\sum_{i=1}^nU\xi_i$, $UZ\sim N(0,U\Sigma U^\top)$, $\delta_\mathcal{B}(UW,U\Sigma U^\top)=\delta_\mathcal{B}(W,\Sigma)$ and $UA\in\mathcal{B}$ for all $A\in\mathcal{B}$. Therefore, it is enough to prove \eqref{ball1} when $\Sigma$ is diagonal with positive diagonal entries, which we assume below. 
The proof is a combination of a smoothing argument and a Gaussian anti-concentration inequality for ellipsoids by \cite{Zh20}, the decomposition \eqref{n18}, and a recursive argument by \cite{Ra19a}. 

Fix $\beta_0>0$. Define
\ben{\label{n11b}
K'(\beta_0)=\sup_{W,\Sigma} \frac{\sup_{A\in \mathcal{B}}|\P(W\in A)-\P(Z\in A)|}{\max \left\{\beta_0, \Psi\left(\delta_{\mathcal{B}}(W,\Sigma)\right)\right\}},
}
where $Z\sim N(0,\Sigma)$ and the first supremum is taken over the family of all sums $W=\sum_{i=1}^n \xi_i$ of $n$ independent centered random vectors with $\E |\xi_i|^4<\infty$, and over diagonal matrices $\Sigma$ with positive entries.
We will obtain a recursive inequality for $K'(\beta_0)$ and prove that 
\ben{\label{n22b}
K'(\beta_0)\leq  C
}
for an absolute constant $C$ that does not depend on $\beta_0$. Eq.\eq{ball1} then follows by sending $\beta_0\to 0$.

Now we fix a $W=\sum_{i=1}^n \xi_i$, $n\geq 1$, and $\Sigma$ in the aforementioned family (will take sup in \eq{n31}). Let 
\ben{\label{n12b}
\bar \beta=\max \left\{\beta_0,  \Psi\left(\delta_\mathcal{B}(W,\Sigma)\right)\right\}.
}

\begin{lemma}[Lemma A.3 of \cite{Zh20}]\label{nl1b}
For any $A\in \mathcal{B}$ and $\eps>0$, there exists a $C^\infty$ function $\tilde h_{A,\eps}$ (which depends only on $A$ and $\eps$) such that
\ben{\label{n06b-1}
\tilde h_{A,\eps}(x)=1\ \text{for}\ x\in A,\qquad \tilde h_{A,\eps}(x)=0\ \text{for}\ x\in \mathbb{R}^d\backslash A^\eps,\qquad 0\leq \tilde h_{A,\eps}(x)\leq 1,
}
and
\ben{\label{n06b}
M_r(\tilde h_{A,\eps})\leq \frac{C}{\eps^r}\qquad\text{for }r=1,2,3,4,
}
where $C$ is an absolute constant that does not depend on $A$ and $\eps$. 
\end{lemma}

\begin{lemma}[Lemma A.4 of \cite{Zh20}]\label{lem4b}
For any $d$-dimensional random vector $W$ and any $\eps>0$,
\ben{\label{n17b}
  \sup_{A\in \mathcal{B}} |\P(W\in A)-\P(Z\in A)| \leq \sup_{A\in\mathcal{B}}\P(Z\in A^\eps\setminus A) +
    \sup_{A\in\mathcal{B}}\babs{\E\tilde h_{A,\eps}(W)-\E\tilde h_{A,\eps}(Z)},
}
where $\tilde h_{A,\eps}$ is as in Lemma \ref{nl1b}. 
\end{lemma}

%
Recall that we assumed without loss of generality that $\Sigma$ is diagonal with positive diagonal entries.
We write $\sigma_j$ for the $j$-th diagonal entry of $\Sigma^{1/2}$. Set $\tilde\sigma:=\min\{\sigma_1,\dots,\sigma_d\}$. 
\begin{lemma}[Lemma A.2 of \cite{Zh20}]\label{anti-ball}
For any $\eps>0$,
\be{
\sup_{A\in\mathcal{B}}\P(Z\in A^\eps\setminus A)\leq C\tilde\sigma^{-1}\eps.
}
\end{lemma}

We now fix $A\in \mathcal{B}$ (will take sup in \eq{n30}), $0<\eps\leq \tilde \sigma$, write $h:=\tilde h_{A, \eps}$ and proceed to bound $|\E h(W)-\E h(Z)|$ by the decomposition \eq{n18}. 
Consider the solution $f$ to the Stein equation \eqref{n03}, which is given by \eqref{sol}. 
Since $h$ has bounded partial derivatives up to the third order, $f$ is thrice differentiable with bounded partial derivatives. Using the integration by parts formula, we have for $1\leq j, k, l\leq d$ and any $0\leq c_0\leq 1$ that
\besn{   \label{n05b}
  \partial_{jk} f (w) & =
  \int_0^{c_0}  \frac{1}{2\sqrt{s}} \int_{\IR^d}
     \partial_j h(\sqrt{1-s}w+\sqrt{s}\Sigma^{1/2}z) \sigma_k^{-1}\partial_k \phi (z) dz ds\\
   &\quad +\int_{c_0}^1- \frac{1}{2s}
        \int_{\IR^d} h (\sqrt{1-s}w+\sqrt{s}\Sigma^{1/2}z) \sigma_j^{-1}\sigma_k^{-1}\partial_{jk}\phi (z) dz ds
}
and
\besn{\label{n10b}
  \partial_{jkl} f (w) & =
  \int_0^{c_0} \frac{\sqrt{1-s}}{2\sqrt{s}} \int_{\IR^d}
     \partial_{jk} h (\sqrt{1-s}w+\sqrt{s}\Sigma^{1/2}z) \sigma_l^{-1}\partial_l \phi (z) dzds\\
     &\quad +\int_{c_0}^1 \frac{\sqrt{1-s}}{2s^{3/2}}
        \int_{\IR^d} h (\sqrt{1-s}w+\sqrt{s}\Sigma^{1/2}z)\sigma_j^{-1}\sigma_k^{-1}\sigma_l^{-1} \partial_{jkl} \phi (z) dz ds.
}

We first bound $R_1$ in \eqref{n25}. Using the expression of $\partial_{jk} f$ in \eq{n05b} with $c_0=(\eps/\tilde{\sigma} )^2$, we have
\be{
R_1=R_{11}+R_{12},
}
where
\be{
R_{11}= \sum_{j,k=1}^d\E\left[ V_{jk} \int_0^{(\eps/\tilde{\sigma})^2} \frac{1}{2\sqrt{s}} \int_{\IR^d}\partial_j h(\sqrt{1-s}W+\sqrt{s}\Sigma^{1/2} z)\sigma_k^{-1}\partial_k \phi (z)dzds\right]
}
and 
\be{
R_{12}=\sum_{j,k=1}^d\E\left[ V_{jk}\int_{(\eps/\tilde{\sigma})^2}^1 -\frac{1}{2s} \int_{\IR^d} h(\sqrt{1-s}W+\sqrt{s} \Sigma^{1/2}z)\sigma_j^{-1}\sigma_k^{-1}\partial_{jk} \phi (z)dzds\right].
}
For $R_{11}$, applying analogous arguments to \eqref{r11}, we obtain
\[
|R_{11}|\leq C\tilde{\sigma}^{-1}\left\{ \sum_{j,k=1}^d \sigma_k^{-2}\E [V_{jk}^2] \right\}^{1/2}
\leq C\tilde{\sigma}^{-2}\left\{ \sum_{j,k=1}^d \E [V_{jk}^2] \right\}^{1/2}.
\]
The triangle inequality yields, for $V$ in \eq{v-eq},
\begin{align*}
\left\{ \sum_{j,k=1}^d \E [V_{jk}^2] \right\}^{1/2}
\leq\|\Sigma-\Var(W)\|_{H.S.}
+\frac{1}{2}\left\{ \sum_{j,k=1}^d\Var \left[\sum_{i=1}^n\xi_{ij}\xi_{ik}\right]\right\}^{1/2}.
\end{align*}
Therefore, we deduce from \eqref{v-est}
\[
|R_{11}|\leq C\delta_\mathcal{B}(W,\Sigma).
\]
For $R_{12}$, we apply analogous arguments to \eqref{r12} and obtain
\[
|R_{12}|\leq  C|\log(\eps/\tilde{\sigma})|\left\{ \sum_{j,k=1}^d (\sigma_j\sigma_k)^{-2}\E [V_{jk}^2] \right\}^{1/2}
\leq C|\log (\eps/\tilde{\sigma})|\delta_\mathcal{B}(W,\Sigma).
\]
Therefore,
\ben{\label{n19b}
|R_1|\leq C(|\log(\eps/\tilde{\sigma})|\vee 1)\delta_\mathcal{B}(W,\Sigma).
}
Next, we bound $R_2$ in \eqref{n09}. 
Using the expression of $\partial_{jkl} f$ in \eq{n10b} with $c_0=(\eps/\tilde{\sigma})^2$, we have
\be{
R_2=R_{21}+R_{22},
}
where
\begin{multline*}
R_{21}=\frac{1}{4} \sum_{i=1}^n \sum_{j,k,l=1}^d \E \Big\{ U(\xi_{ij}'-\xi_{ij})(\xi_{ik}'-\xi_{ik})(\xi_{il}'-\xi_{il})
\int_0^{(\eps/\tilde{\sigma})^2} \frac{\sqrt{1-s}}{2\sqrt{s}} \\
\times \int_{\IR^d} [\partial_{jk} h(\sqrt{1-s}(W+(1-U)(\xi_i'-\xi_i))+\sqrt{s}\Sigma^{1/2}z)-\partial_{jk} h(\sqrt{1-s}(W+U(\xi_i'-\xi_i))+\sqrt{s}\Sigma^{1/2}z)]\\
\times \sigma_l^{-1}\partial_l \phi(z) dz ds \Big\}
\end{multline*}
and
\begin{multline*}
R_{22}=\frac{1}{4} \sum_{i=1}^n \sum_{j,k,l=1}^d \E \Big\{ U(\xi_{ij}'-\xi_{ij})(\xi_{ik}'-\xi_{ik})(\xi_{il}'-\xi_{il})
\int_{(\eps/\tilde{\sigma})^2}^1 \frac{\sqrt{1-s}}{2s^{3/2}} \\
 \times \int_{\IR^d} [h(\sqrt{1-s}(W+(1-U)(\xi_i'-\xi_i))+\sqrt{s}\Sigma^{1/2}z)-h(\sqrt{1-s}(W+U(\xi_i'-\xi_i))+\sqrt{s}\Sigma^{1/2}z)]\\
 \times \sigma_j^{-1}\sigma_k^{-1}\sigma_l^{-1}\partial_{jkl} \phi(z) dz ds \Big\}.
\end{multline*}
Note that here we treat $R_2$ differently from the treatment for the convex set distance because $h$ now is smoother.
Let $W^{(i)}=W-\xi_i$ for $i\in \{1,\dots, n\}$. We will use the mean value theorem for the differences involving $h$ in the above two expressions as in \eq{n32}, the fact that $\nabla h$ is non-zero only in $A^\eps\backslash A$ and bound
\be{
\P(\sqrt{1-s} W^{(i)}\in A_i^\eps\backslash A_i \mid U, U', \xi_i, \xi_i'),
}
where $0<s<1$, $U'$ is a uniform random variable on $[0,1]$ independent of everything else, and $A_i$ is a Euclidean ball which may depend on $U$, $U'$, $\xi_i$, $\xi_i'$, $s$ and $\Sigma^{1/2}z$.
We have by Lemma \ref{anti-ball}
\besn{\label{n14b}
\P(\sqrt{1-s} W^{(i)}\in A_i^\eps\backslash A_i \mid U, U', \xi_i, \xi_i')
\leq  C\frac{\eps}{\tilde\sigma \sqrt{1-s}} + 2\sup_{A\in \mathcal{B}} \left|\P(W^{(i)} \in A) -P(Z\in A)    \right|.
}
From \eq{n11b}, and regarding $W^{(i)}=\sum_{j: j\ne i} \xi_j+0$ as a sum of $n$ independent centered random vectors,  we have
\besn{\label{n13b}
\sup_{A\in \mathcal{B}} \left|\P(W^{(i)} \in A) -P(Z\in A)    \right|
\leq  K'(\beta_0) \max \left\{ \beta_0,  \Psi(\delta_\mathcal{B}(W^{(i)},\Sigma))\right\}.
}
Since
\ba{
\|\Var(W)-\Var(W^{(i)})\|_{H.S.}
&=\sqrt{\sum_{j,k=1}^d(\E[\xi_{ij}\xi_{ik}])^2}
\leq \sqrt{\E[|\xi_i|^4]}
}
and $\sqrt{x}+\sqrt{y}\leq\sqrt{2(x+y)}$ for any $x,y\geq0$, we have
\ba{
\delta_\mathcal{B}(W^{(i)},\Sigma)
&\leq\|\Sigma^{-1}\|_{op}\left\{\|\Sigma-\Var(W)\|_{H.S.}
+\sqrt{\E[|\xi_i|^4]}
+\sqrt{\sum_{j=1\atop j\ne i}^n \E[|\xi_j|^4]}\right\}\\
&\leq\|\Sigma^{-1}\|_{op}\left\{\|\Sigma-\Var(W)\|_{H.S.}+\sqrt{2\sum_{j=1}^n \E[|\xi_j|^4]}\right\}
\leq\sqrt{2}\delta(W,\Sigma).
} 
Hence, we obtain by Lemma \ref{lem:psi} 
\ba{
\Psi(\delta(W^{(i)},\Sigma))\leq C\Psi(\delta(W,\Sigma))\leq C\bar\beta.
}
Thus we conclude
\ben{\label{n15b}
\sup_{A\in \mathcal{B}} \left|\P(W^{(i)} \in A) -P(Z\in A)    \right|\leq K'(\beta_0)\max \left\{ \beta_0,  C\bar\beta\right\}
\leq C K'(\beta_0)\bar\beta.
}
Using the mean value theorem for $R_{21},R_{22}$ and applying \eq{n06b}, \eq{n14b}, \eq{n15b} and \eqref{raic4.9}, we have
\ben{\label{n16-2b}
|R_{21}|+|R_{22}|\leq \frac{C}{\tilde{\sigma}^{2}\eps^2} \sum_{i=1}^n \E[|\xi_i|^4] \left(\frac{\eps}{\tilde{\sigma}} +K'(\beta_0)\bar \beta  \right).
}
From Lemmas \ref{lem4b}--\ref{anti-ball}, \eq{n18}, \eq{n19b}, \eq{n16-2b}, we have
\besn{\label{n30}
&\sup_{A\in \mathcal{B}} |P(W\in A)-P(Z\in A)|\\
 \leq & C\tilde{\sigma}^{-1}\eps+C(|\log (\eps/\tilde{\sigma})|\vee 1)\delta_{\mathcal{B}}(W,\Sigma)
+\frac{C\tilde{\sigma}^{2} \Psi^2(\delta_\mathcal{B}(W,\Sigma))}{\eps^2}  \left(\frac{\eps}{\tilde{\sigma}} +K'(\beta_0)\bar \beta \right).
}
Choose $\eps=\min\{ \sqrt{2C} \tilde{\sigma}\Psi(\delta_\mathcal{B}(W, \Sigma)), \tilde \sigma\} $ with the same absolute constant $C$ as in the third term on the right-hand side of \eq{n30}. If $\eps<\tilde \sigma$ and $\bar \beta\leq 0.5$, then from \eq{n30},
\be{
\sup_{A\in \mathcal{B}} |\P(W\in A)-\P(Z\in A)|\leq \left( C +\frac{K'(\beta_0)}{2} \right) \bar \beta;
}
hence
\ben{\label{n20b}
\frac{\sup_{A\in \mathcal{B}} |\P(W\in A)-\P(Z\in A)|}{\bar \beta}\leq C +\frac{K'(\beta_0)}{2}.
}
If $\bar \beta>0.5$ or if $\eps=\tilde \sigma$, then $\bar \beta$ is bounded away from 0 by an absolute constant; hence
\be{
\frac{\sup_{A\in \mathcal{A}} |\P(W\in A)-\P(Z\in A)|}{\bar \beta}\leq \frac{1}{\bar \beta}\leq C.
}
Note that the right-hand sides of the above two bounds do not depend on $W$ or $\Sigma$. Taking supremum over $W$ and $\Sigma$, we obtain
\ben{\label{n31}
K'(\beta_0)\leq C+\frac{K'(\beta_0)}{2}.
}
This implies $\eq{n22b}$, hence \eq{ball1}.

\subsection{Proof of Theorem \ref{BALL2}}

In this proof, we use a new smoothing of the indicator function of centered balls to apply a Gaussian anti-concentration inequality of \cite{GNSU19} (see Lemma \ref{anti-ball2}). The exchangeable pair approach and the symmetry argument used in the proofs of Theorems \ref{convex} and \ref{BALL} no longer help and we use Stein's leave-one-out trick in this proof.

Set $W^{(i)}:=W-\xi_i$ for $i=1,\dots,n$. 
The following simple lemma plays a key role in our proof (see \cite[Proposition A.3]{Po88} and \cite[Lemma 1]{PeSc18} for related results). 
\begin{lemma}\label{portnoy}
For every $i=1,\dots,n$,
\[
\E[(\xi_{i}\cdot W^{(i)})^2\mid\xi_i]\leq\xi_i^\top\Sigma_W\xi_i
\]
and
\[
\E[(\xi_i\cdot W^{(i)})^4\mid\xi_i]\leq|\xi_i|^4\sum_{u=1}^n\E[|\xi_u|^4]+3(\xi_i^\top\Sigma_W\xi_i)^2.
\]
\end{lemma}

\begin{proof}
Note that $\xi_1,\dots,\xi_n$ are independent and $\xi_i\cdot W^{(i)}=\sum_{u:u\neq i}\xi_i\cdot\xi_u$.  
Thus, a direct computation shows
\ba{
\E[(\xi_{i}\cdot W^{(i)})^2\mid\xi_i]
=\sum_{u:u\neq i}\E[(\xi_i\cdot\xi_{u})^2\mid\xi_i]
}
and
\ba{
\E[(\xi_i\cdot W^{(i)})^4\mid\xi_i]
&=\sum_{u:u\neq i}\E[(\xi_i\cdot\xi_u)^4\mid\xi_i]
+3\sum_{u,v:u,v\neq i,u\neq v}\E[(\xi_i\cdot\xi_u)^2\mid\xi_i]\E[(\xi_i\cdot\xi_v)^2\mid\xi_i].
}
The Schwarz inequality yields
\ba{
\sum_{u:u\neq i}\E[(\xi_i\cdot\xi_u)^4\mid\xi_i]
\leq\sum_{u:u\neq i}|\xi_i|^4\E[|\xi_u|^4]
\leq|\xi_i|^4\sum_{u=1}^n\E[|\xi_u|^4].
}
Also, noting that $\E[\xi_{u}\xi_{u}^\top]$ is positive semidefinite for every $u$, we have
\ba{
\sum_{u:u\neq i}\E[(\xi_i\cdot\xi_{u})^2\mid\xi_i]
&=\sum_{u:u\neq i}\E[(\xi_i^\top\xi_u)(\xi_u^\top\xi_i)\mid\xi_i]
=\sum_{u:u\neq i}\xi_i^\top\E[\xi_{u}\xi_{u}^\top]\xi_i\\
&\leq\sum_{u=1}^n\xi_i^\top\E[\xi_{u}\xi_{u}^\top]\xi_i
=\xi_i^\top\Sigma_W\xi_i
}
and
\ba{
\sum_{u,v:u,v\neq i,u\neq v}\E[(\xi_i\cdot\xi_u)^2\mid\xi_i]\E[(\xi_i\cdot\xi_v)^2\mid\xi_i]
&\leq\sum_{u,v:u,v\neq i}\E[(\xi_i\cdot\xi_u)^2\mid\xi_i]\E[(\xi_i\cdot\xi_v)^2\mid\xi_i]\\
&=\left(\sum_{u:u\neq i}\E[(\xi_i\cdot\xi_u)^2\mid\xi_i]\right)^2
\leq(\xi_i^\top\Sigma_W\xi_i)^2.
}
This completes the proof. 
\end{proof}

We will also use the following sharp anti-concentration inequality for the squared norm of a Gaussian vector established in \cite{GNSU19}. 
Recall $\Lambda_2(\Sigma)$ and $\varkappa(\Sigma)$ above the statement of Theorem \ref{BALL2}.

\begin{lemma}[Theorem 2.7 of \cite{GNSU19}]\label{anti-ball2}
If $\Lambda_2(\Sigma)>0$, then
\[
\sup_{a\geq0}\P(a\leq |Z+\mu|^2\leq a+\eps)\leq C\varkappa(\Sigma)\eps
\]
for any $\mu\in\mathbb R^d$ and $\eps>0$.  
\end{lemma}

We turn to the main body of the proof. 
First, without loss of generality, we may assume $\xi_1,\dots,\xi_n$ and $Z$ are independent. 
Fix a $C^\infty$ function $g:\mathbb{R}\to[0,1]$ such that $g(x)=1$ for $x\leq0$ and $g(x)=0$ for $x\geq1$. For any $a\in\mathbb R$ and $\eps>0$, we define the function $h_{a,\eps}:\mathbb{R}^d\to[0,1]$ by 
\[
h_{a,\eps}(x)=g(\eps^{-1}(|x|^2-a)),\qquad x\in\mathbb R^d.
\]
Then we have
\ba{
\P(|W|\leq r)
&\leq\E h_{r^2,\eps}(W)
\leq|\E h_{r^2,\eps}(W)-\E h_{r^2,\eps}(Z)|
+\E h_{r^2,\eps}(Z)\\
&\leq|\E h_{r^2,\eps}(W)-\E h_{r^2,\eps}(Z)|
+\P(|Z|^2\leq r^2+\eps)
}
and
\ba{
\P(|W|\leq r)
&=\P(|W|^2-r^2\leq0)
\geq\E h_{r^2-\eps,\eps}(W)\\
&\geq-|\E h_{r^2-\eps,\eps}(W)-\E h_{r^2-\eps,\eps}(Z)|
+\E h_{r^2-\eps,\eps}(Z)\\
&\geq-|\E h_{r^2-\eps,\eps}(W)-\E h_{r^2-\eps,\eps}(Z)|
+\P(|Z|^2\leq r^2-\eps).
}
Thus we obtain
\ba{
\sup_{r\geq0}|\P(|W|\leq r)-\P(|Z|\leq r)|
\leq\sup_{a\in\mathbb R}|\E h_{a,\eps}(W)-\E h_{a,\eps}(Z)|
+\sup_{a\geq0}\P(a<|Z|^2\leq a+\eps).
}
Applying Lemma \ref{anti-ball2} to the second term on the right hand side, we infer that
\ben{\label{eq:smoothing}
\sup_{r\geq0}|\P(|W|\leq r)-\P(|Z|\leq r)|
\leq\sup_{a\in\mathbb R}|\E h_{a,\eps}(W)-\E h_{a,\eps}(Z)|
+C\varkappa(\Sigma)\eps.
}

Fix $a\in\mathbb R$ and $\eps>0$, write $h=h_{a,\eps}$ and proceed to bound $|\E h(W)-\E h(Z)|$. To accomplish this, we decompose $\E h(W)-\E h(Z)$ as
\[
\E h(W)-\E h(Z)=\{\E h(W)-\E h(Z_W)\}+\{\E h(Z_W)-\E h(Z)\},
\]
where $Z_W$ is a centered Gaussian vector in $\mathbb R^d$ with covariance matrix $\Sigma_W$ which is independent of everything else. We will show the following bounds:
\ban{
|\E h(W)-\E h(Z_W)|&\leq C\{\eps^{-3}\delta_1(W)+\eps^{-2}\delta_2(W)\},\label{ball2-bound}\\
|\E h(Z_W)-\E h(Z)|&\leq C\{\eps^{-2}\delta_0(\Sigma_W,\Sigma)
+\eps^{-1}\delta_0'(\Sigma_W,\Sigma)\}.\label{ball2-bound-g}
}
Then, inserting these bounds into \eqref{eq:smoothing}, we obtain
\bm{
\sup_{r\geq0}|\P(|W|\leq r)-\P(|Z|\leq r)|\\
\leq C\{\eps^{-3}\delta_1(W)+\eps^{-2}\delta_2(W)
+\eps^{-2}\delta_0(\Sigma_W,\Sigma)
+\eps^{-1}\delta_0'(\Sigma_W,\Sigma)
+\varkappa(\Sigma)\eps\}.
}
The desired result then follows by setting
\[
\eps=(\delta_1(W)/\varkappa(\Sigma))^{1/4}+(\delta_2(W)/\varkappa(\Sigma))^{1/3}
+(\delta_0(\Sigma_W,\Sigma)/\varkappa(\Sigma))^{1/3}+(\delta'_0(\Sigma_W,\Sigma)/\varkappa(\Sigma))^{1/2}.
\]

Now we prove \eqref{ball2-bound} at first. 
We consider the Stein equation associated with $h$ and covariance matrix $\Sigma_W$:
\ben{   \label{n03c}
  \langle\Hess f(w),\Sigma_W\rangle_{H.S.}-w\cdot \nabla f(w)=h(w)-\E h(Z_W),\qquad w\in \IR^d.
}
It can be verified directly 
that (cf. \eq{sol})
\be{
f(w)=\int_0^1-\frac{1}{2(1-s)}\E[h(\sqrt{1-s}w+\sqrt{s}Z_W)-\E h (Z_W)]ds
}
is a solution to \eq{n03c}. 
Hence we have
\ben{\label{stein-eq}
\E h(W)-\E h(Z_W)=\E\langle\Hess f(W),\Sigma_W\rangle_{H.S.}-\E [W\cdot \nabla f(W)].
}
Also, we have by differentiation under the integral sign
\besn{\label{n10c}
  \nabla^2 f (w) & =
  \int_0^1 -\frac{1}{2}\E[\nabla^2 h (\sqrt{1-s}w+\sqrt{s}Z_W)]ds.
}
To evaluate the right hand side of \eqref{stein-eq}, we employ Stein's leave-one-out trick. 
First, we have
\ba{
\E [W\cdot \nabla f(W)]
=\sum_{i=1}^n\sum_{j=1}^d\E[\xi_{ij}\partial_jf(W)].
}
Taylor expanding $\partial_jf(W)$ around $W^{(i)}$, we obtain
\ba{
\E [W\cdot \nabla f(W)]
=\sum_{i=1}^n\sum_{j=1}^d\E[\xi_{ij}\partial_jf(W^{(i)})]
+\sum_{i=1}^n\sum_{j,k=1}^d\E[\xi_{ij}\xi_{ik}\partial_{jk}f(W^{(i)})]
+R_1,
}
where
\[
R_1=\sum_{i=1}^n\sum_{j,k=1}^d\E[\xi_{ij}\xi_{ik}\{\partial_{jk}f(W^{(i)}+U\xi_i)-\partial_{jk}f(W^{(i)})\}]
\]
and $U$ is a uniform random variable on $[0,1]$ independent of everything else. 
Since $\xi_i$ and $W^{(i)}$ are independent and $\E\xi_i=0$, we deduce
\ben{\label{loo-1}
\E [W\cdot \nabla f(W)]
=\sum_{i=1}^n\sum_{j,k=1}^d\E[\xi_{ij}\xi_{ik}]\E[\partial_{jk}f(W^{(i)})]
+R_1.
}
Next, since $\Sigma_W=\E[WW^\top]$, we have
\ban{
\E\langle\Hess f(W),\Sigma_W\rangle_{H.S.}
&=\sum_{j,k=1}^d\E[W_jW_k]\E[\partial_{jk}f(W)]
=\sum_{i=1}^n\sum_{j,k=1}^d\E[\xi_{ij}\xi_{ik}]\E[\partial_{jk}f(W)]
\nonumber\\
&=\sum_{i=1}^n\sum_{j,k=1}^d\E[\xi_{ij}\xi_{ik}]\E[\partial_{jk}f(W^{(i)})]
+R_2,\label{loo-2}
}
where
\[
R_2=\sum_{i=1}^n\sum_{j,k=1}^d\E[\xi_{ij}\xi_{ik}]\E[\partial_{jk}f(W)-\partial_{jk}f(W^{(i)})].
\]
Combining \eqref{stein-eq}, \eqref{loo-1} and \eqref{loo-2}, we conclude
\ben{\label{basic-bound}
|\E h(W)-\E h(Z_W)|\leq |R_1|+|R_2|.
}
Now we bound $R_1$ and $R_2$. 
By definition we have
\ben{\label{h-deriv2}
\partial_{jk} h(w)=4\eps^{-2}g''(\eps^{-1}(|w|^2-a))w_jw_k+2\eps^{-1}g'(\eps^{-1}(|w|^2-a))\delta_{jk}.
}
Hence we obtain from \eqref{n10c}
\ban{
R_1&=-\sum_{i=1}^n\int_0^{1} \frac{1}{2}\sum_{j,k=1}^d\E\Big\{ \xi_{ij}\xi_{ik}\Big[ 4\eps^{-2}g''(\eps^{-1}(| W^{(i),s}+\tilde\xi^s_i|^2-a))(W^{(i),s}_j+\tilde\xi^s_{ij})(W^{(i),s}_k+\tilde\xi^s_{ik})
\nonumber\\
&\hphantom{=-\sum_{i=1}^n\int_0^{1} \frac{1}{2}\sum_{j,k,l=1}^d\E[}
-4\eps^{-2}g''(\eps^{-1}(|W^{(i),s}|^2-a))W^{(i),s}_jW^{(i),s}_k
\nonumber\\
&\hphantom{=-\sum_{i=1}^n\int_0^{1} \frac{1}{2}\sum_{j,k,l=1}^d\E[}
+2\eps^{-1}g'(\eps^{-1}(|W^{(i),s}+\tilde\xi^s_i|^2-a))\delta_{jk}
\nonumber\\
&\hphantom{=-\sum_{i=1}^n\int_0^{1} \frac{1}{2}\sum_{j,k,l=1}^d\E[}
-2\eps^{-1}g'(\eps^{-1}(|W^{(i),s}|^2-a))\delta_{jk} \Big] \Big\}ds
\nonumber\\
&=:R_{11}+R_{12}+R_{13}+R_{14},
\label{r1-decomp}
}
where $W^{(i),s}:=\sqrt{1-s}W^{(i)}+\sqrt sZ_W$, $\tilde\xi^s_i:=\sqrt{1-s}U\xi_i$ and
\ba{
R_{11}&:=-2\eps^{-2}\sum_{i=1}^n\int_0^{1}(1-s)\sum_{j,k=1}^d\E[U^2\xi_{ij}^2\xi_{ik}^2g''(\eps^{-1}(| W^{(i),s}+\tilde\xi^s_i|^2-a))]ds,\\
R_{12}&:=-4\eps^{-2}\sum_{i=1}^n\int_0^{1}\sqrt{1-s}\sum_{j,k=1}^d\E[U\xi_{ij}^2\xi_{ik}W^{(i),s}_kg''(\eps^{-1}(| W^{(i),s}+\tilde\xi^s_i|^2-a))]ds,\\
R_{13}&:=-2\eps^{-2}\sum_{i=1}^n\int_0^{1}\sum_{j,k=1}^d\E[\xi_{ij}\xi_{ik}W^{(i),s}_jW^{(i),s}_k\{g''(\eps^{-1}(| W^{(i),s}+\tilde\xi^s_i|^2-a))-g''(\eps^{-1}(| W^{(i),s}|^2-a))\}]ds,\\
R_{14}&:=-\eps^{-1}\sum_{i=1}^n\int_0^{1}\sum_{j=1}^d\E[\xi_{ij}^2\{g'(\eps^{-1}(| W^{(i),s}+\tilde\xi^s_i|^2-a))-g'(\eps^{-1}(| W^{(i),s}|^2-a))\}]ds.
} 
We have
\ben{\label{r11-bound}
|R_{11}|
\leq C\eps^{-2}\sum_{i=1}^n\E[|\xi_i|^4]
}
and
\ba{
|R_{12}|
\leq C\eps^{-2}\sum_{i=1}^n\E[|\xi_i|^2(|\xi_{i}\cdot W^{(i)}|+|\xi_i\cdot Z_W|)].
}
Meanwhile, using the fundamental theorem of calculus, we can rewrite $R_{13}$ and $R_{14}$ as
\ba{
R_{13}=-4\eps^{-3}\sum_{i=1}^n\int_0^{1}\E[(\xi_{i}\cdot W^{(i),s})^2(W^{(i),s}+U'\tilde\xi^s_i)\cdot \tilde\xi^s_i g^{(3)}(\eps^{-1}(| W^{(i),s}+U'\tilde\xi^s_i|^2-a))]ds
}
and
\ba{
R_{14}=-2\eps^{-2}\sum_{i=1}^n\int_0^{1}\E[|\xi_{i}|^2(W^{(i),s}+U'\tilde\xi^s_i)\cdot \tilde\xi^s_i g''(\eps^{-1}(| W^{(i),s}+U'\tilde\xi^s_i|^2-a))]ds,
}
where $U'$ is a uniform random variable on $[0,1]$ independent of everything else. 
Hence we obtain
\ba{
|R_{13}|
&\leq C\eps^{-3}\sum_{i=1}^n\int_0^{1}\E[(\xi_{i}\cdot W^{(i),s})^2(|W^{(i),s}\cdot\xi_i|+|\xi_i|^2)]ds\\
&\leq C\eps^{-3}\sum_{i=1}^n\E[|\xi_{i}\cdot W^{(i)}|^3+|\xi_{i}\cdot Z_W|^3
+(\xi_{i}\cdot W^{(i)})^2|\xi_i|^2+(\xi_{i}\cdot Z_W)^2|\xi_i|^2]
}
and
\ba{
|R_{14}|
&\leq C\eps^{-2}\sum_{i=1}^n\int_0^{1}\E[|\xi_{i}|^2(|W^{(i),s}\cdot\xi_i|+|\xi_i|^2)]ds\\
&\leq C\eps^{-2}\sum_{i=1}^n\E[|\xi_{i}|^2|\xi_{i}\cdot W^{(i)}|+|\xi_i|^2|\xi_{i}\cdot Z_W|+|\xi_i|^4].
}
The Schwarz inequality and Lemma \ref{portnoy} imply that
\ba{
\E[|\xi_{i}\cdot W^{(i)}|^3\mid\xi_i]
&=\E[|\xi_{i}\cdot W^{(i)}|\cdot|\xi_{i}\cdot W^{(i)}|^2\mid\xi_i]\\
&\leq\sqrt{\E[|\xi_{i}\cdot W^{(i)}|^2\mid\xi_i]\E[|\xi_{i}\cdot W^{(i)}|^4\mid\xi_i]}\\
&\leq\sqrt{(\xi_i^\top\Sigma_W\xi_i)\left(|\xi_i|^4\sum_{u=1}^n\E[|\xi_u|^4]+3(\xi_i^\top\Sigma_W\xi_i)^2\right)}\\
&\leq\sqrt{\xi_i^\top\Sigma_W\xi_i}|\xi_i|^2\sqrt{\sum_{u=1}^n\E[|\xi_u|^4]}+\sqrt{3}(\xi_i^\top\Sigma_W\xi_i)^{3/2}.
}
Hence we have
\[
\sum_{i=1}^n\E[|\xi_{i}\cdot W^{(i)}|^3]
\leq\sum_{i=1}^n\E\left[\sqrt{\xi_i^\top\Sigma_W\xi_i}|\xi_i|^2\right]\sqrt{\sum_{u=1}^n\E[|\xi_u|^4]}+\sqrt{3}\sum_{i=1}^n\E[(\xi_i^\top\Sigma_W\xi_i)^{3/2}].
\]
The Schwarz inequality yields
\[
\sum_{i=1}^n\E\left[\sqrt{\xi_i^\top\Sigma_W\xi_i}|\xi_i|^2\right]
\leq\sqrt{\sum_{i=1}^n\E\left[\xi_i^\top\Sigma_W\xi_i\right]\sum_{i=1}^n\E[|\xi_i|^4]}.
\]
Since 
\ba{
\sum_{i=1}^n\E\left[\xi_i^\top\Sigma_W\xi_i\right]
&=\sum_{i=1}^n\E\left[\tr(\xi_i^\top\Sigma_W\xi_i)\right]
=\sum_{i=1}^n\E\left[\tr(\Sigma_W\xi_i\xi_i^\top)\right]\\
&=\tr\left(\Sigma_W\sum_{i=1}^n\E[\xi_i\xi_i^\top]\right)
=\tr(\Sigma_W^2)=\|\Sigma_W\|_{H.S.}^2,
}
we conclude that
\[
\sum_{i=1}^n\E[|\xi_{i}\cdot W^{(i)}|^3]
\leq\|\Sigma_W\|_{H.S.}\sum_{i=1}^n\E[|\xi_i|^4]+\sqrt{3}\sum_{i=1}^n\E[(\xi_i^\top\Sigma_W\xi_i)^{3/2}].
\]
The Schwarz inequality and Lemma \ref{portnoy} also imply that
\ba{
\E[|\xi_i|^2|\xi_{i}\cdot W^{(i)}|]
\leq\E\left[|\xi_i|^2\sqrt{\E[|\xi_{i}\cdot W^{(i)}|^2\mid\xi_i]}\right]
\leq\E\left[|\xi_i|^2\sqrt{\xi_i^\top\Sigma_W\xi_i}\right]
}
and
\ba{
\E[|\xi_{i}\cdot W^{(i)}|^2|\xi_i|^2]
\leq\E[(\xi_i^\top\Sigma_W\xi_i)|\xi_i|^2].
}
In addition, conditional on $\xi_i$, $\xi_i\cdot Z_W$ follows the normal distribution with mean 0 and variance $\xi_i^\top\Sigma_W\xi_i$. Hence we obtain
\[
\E[|\xi_i\cdot Z_W|^3]\leq C\E[(\xi_i^\top\Sigma_W\xi_i)^{3/2}],\qquad
\E[(\xi_i\cdot Z_W)^2|\xi_i|^2]\leq C\E[(\xi_i^\top\Sigma_W\xi_i)|\xi_i|^2]
\]
and
\[
\E[|\xi_i|^2|\xi_i\cdot Z_W|]\leq C\E\left[|\xi_i|^2\sqrt{\xi_i^\top\Sigma_W\xi_i}\right].
\]
Consequently, we deduce
\ben{\label{r12-bound}
|R_{12}|
\leq C\eps^{-2}\sum_{i=1}^n\E\left[|\xi_i|^2\sqrt{\xi_i^\top\Sigma_W\xi_i}\right]
}
and
\ben{\label{r13-bound}
|R_{13}|\leq C\eps^{-3}\left\{\|\Sigma_W\|_{H.S.}\sum_{i=1}^n\E[|\xi_i|^4]
+\sum_{i=1}^n\E[(\xi_i^\top\Sigma_W\xi_i)^{3/2}]
+\sum_{i=1}^n\E[(\xi_i^\top\Sigma_W\xi_i)|\xi_i|^2]\right\}
}
and
\ben{\label{r14-bound}
|R_{14}|
\leq C\eps^{-2}\sum_{i=1}^n\E\left[|\xi_i|^2\sqrt{\xi_i^\top\Sigma_W\xi_i}+|\xi_i|^4\right].
}
Note that $\xi_i^\top\Sigma_W\xi_i\leq\|\Sigma_W\|_{op}|\xi_i|^2$ and $\|\Sigma_W\|_{op}\leq\|\Sigma_W\|_{H.S.}$. 
Therefore, we deduce from \eqref{r1-decomp}--\eqref{r14-bound} that
\ben{\label{r1-bound}
|R_{1}|\leq C(\eps^{-3}\delta_1(W)+\eps^{-2}\delta_2(W)).
}
Besides, note that we can rewrite $R_2$ as
\[
R_2=\sum_{i=1}^n\sum_{j,k=1}^d\E[\xi_{ij}'\xi_{ik}'\{\partial_{jk}f(W^{(i)}+\xi_i)-\partial_{jk}f(W^{(i)})\}],
\]
where $(\xi_i')_{i=1}^n$ is an independent copy of $(\xi_i)_{i=1}^n$. Hence, we can prove by a similar argument to the above
\ben{\label{r2-bound}
|R_2|\leq C(\eps^{-3}\delta_1(W)+\eps^{-2}\delta_2(W)).
}
Combining \eqref{basic-bound}, \eqref{r1-bound} and \eqref{r2-bound}, we obtain \eqref{ball2-bound}.

Next we prove \eqref{ball2-bound-g}. From \eqref{n03c} we have
\be{
\E h(Z)-\E h(Z_W)=\E[\langle\Hess f(Z),\Sigma_W\rangle_{H.S.}]-\E [Z\cdot \nabla f(Z)].
}
The multivariate Stein identity yields
\ba{
\E [Z\cdot \nabla f(Z)]=\E\langle\Hess f(Z),\Sigma\rangle_{H.S.}.
}
So we obtain
\ben{\label{basic-bound-g}
|\E h(Z_W)-\E h(Z)|\leq |\E\langle\Hess f(Z),\Sigma-\Sigma_W\rangle_{H.S.}|.
}
We have by \eqref{n10c} and \eqref{h-deriv2}
\ba{
&\E\langle\Hess f(Z),\Sigma-\Sigma_W\rangle_{H.S.}\\
&=-2\eps^{-2}\int_0^1 \sum_{j,k=1}^d\E[g''(\eps^{-1}(|Z(s)|^2-a))Z(s)_jZ(s)_k(\Sigma_{jk}-\Sigma_{W,jk})]ds\\
&\quad-\eps^{-1}\int_0^1 \sum_{j=1}^d\E[g'(\eps^{-1}(|Z(s)|^2-a))(\Sigma_{jj}-\Sigma_{W,jj})]ds\\
&=:R'_1+R'_2,
}
where $Z(s):=\sqrt{1-s}Z+\sqrt s Z_W$. 
We can bound $R'_2$ as
\ben{\label{r2-bound-g}
|R'_2|\leq C \eps^{-1}\sum_{j=1}^d|\Sigma_{jj}-\Sigma_{W,jj}|.
}
Meanwhile, we can rewrite $R'_1$ as
\ba{
R'_1=-2\eps^{-2}\int_0^1 \E[g''(\eps^{-1}(|Z(s)|^2-a))Z(s)^\top(\Sigma-\Sigma_W)Z(s)]ds.
}
So we obtain by the Schwarz inequality
\ba{
|R'_1|&\leq C\eps^{-2}\int_0^1 \sqrt{\E[|Z(s)|^2]\E[|(\Sigma-\Sigma_W)Z(s)|^2]}ds.
}
We have
\ba{
\E[|Z(s)|^2]
=\tr(\Var(Z(s)))=(1-s)\tr(\Sigma)+s\tr(\Sigma_W)
}
and
\ba{
\E[|(\Sigma-\Sigma_W)Z(s)|^2]
&=\tr((\Sigma-\Sigma_W)\Var(Z(s))(\Sigma-\Sigma_W))\\
&=\|(\Sigma-\Sigma_W)[\Var(Z(s))]^{1/2}\|_{H.S.}^2\\
&\leq\|\Var(Z(s))\|_{op}\|\Sigma-\Sigma_W\|_{H.S.}^2\\
&\leq(\|\Sigma_W\|_{op}+\|\Sigma\|_{op})\|\Sigma-\Sigma_W\|_{H.S.}^2.
}
Thus we obtain
\ben{\label{r1-bound-g}
|R'_1|\leq C\eps^{-2}\sqrt{(\tr(\Sigma_W)+\tr(\Sigma))(\|\Sigma_W\|_{op}+\|\Sigma\|_{op})}\|\Sigma-\Sigma_W\|_{H.S.}.
}
Combining \eqref{basic-bound-g}, \eqref{r2-bound-g} and \eqref{r1-bound-g}, we obtain \eqref{ball2-bound-g}. Thus we complete the proof. 

\subsection{Proof of Proposition \ref{prop1}}

Since $(|Z|^2-d)/\sqrt{2d}$ converges in law to $N(0,1)$ as $d\to\infty$, by \eq{n41}, $(|W|^2-d)/\sqrt{2d}$ also converges in law to $N(0,1)$. Since $W$ has the same law as $\sqrt{V}Z'$ by assumption, where $V:=n^{-1}\sum_{i=1}^ne_i^2$ and $Z'\sim N(0,I_d)$ is independent of $\{e_i\}_{i=1}^\infty$, $(V|Z'|^2-d)/\sqrt{2d}$ should also converge in law to $N(0,1)$. Since
\[
\frac{V|Z'|^2-d}{\sqrt{2d}}
=V\frac{|Z'|^2-d}{\sqrt{2d}}+\sqrt{\frac{d}{2}}(V-1)
=(V-1)\frac{|Z'|^2-d}{\sqrt{2d}}+\frac{|Z'|^2-d}{\sqrt{2d}}+\sqrt{\frac{d}{2}}(V-1)
\]
and the first term converges to 0 in probability, 
\[
\frac{|Z'|^2-d}{\sqrt{2d}}+\sqrt{\frac{d}{2}}(V-1)
\]
must converge in law to $N(0,1)$. In the above expression, the first term converges in law to $N(0,1)$ and the first and second terms are independent, so this implies $\sqrt{d}(V-1)=o_p(1)$ as $n\to\infty$. Since $\sqrt{n}(V-1)$ converges in law to $N(0,\Var(e_1^2))$, we must have $d/n\to0$.

\subsection{Proof of Theorem \ref{efron} and Corollary \ref{coro-ef}}\label{proof:efron}

First we prove Theorem \ref{efron}. Without loss of generality, we may assume $\Delta_n^*\leq1$. This particularly implies
\ben{\label{ef-wlog}
\varkappa^3(\Sigma)\delta_1(W)\leq1.
}
Conditional on $X$, $X_1^*-\bar{X},\dots,X_n^*-\bar{X}$ are i.i.d.~with mean 0 and covariance matrix $\wh{\Sigma}$. Therefore, applying Theorem \ref{BALL2} conditional on $X$, we obtain
\bmn{\label{ef1}
\sup_{r\geq0}|\P(|W^*|\leq r\mid X)-\P(|Z|\leq r)|
\leq C(\varkappa^{3/4}(\Sigma)(\delta^*_1)^{1/4}+\varkappa^{2/3}(\Sigma)(\delta^*_2)^{1/3}\\
+\varkappa^{2/3}(\Sigma)\delta_0^{1/3}(\wh{\Sigma},\Sigma)+\varkappa^{1/2}(\Sigma)\delta_0'^{1/2}(\wh\Sigma,\Sigma)),
}
where, with $\wt X_i:=X_i-\bar X$,
\ba{
\delta_1^*&:=\frac{\|\wh\Sigma\|_{H.S.}}{n^2}\sum_{i=1}^n|\wt X_i|^4
+\frac{\|\wh\Sigma\|_{op}^{3/2}}{n^{3/2}}\sum_{i=1}^n|\wt X_i|^{3},\\
\delta_2^*&:=\frac{\|\wh\Sigma\|_{op}^{1/2}}{n^{3/2}}\sum_{i=1}^n|\wt X_i|^{3}
+\frac{1}{n^{2}}\sum_{i=1}^n|\wt X_i|^4.
} 
Set $\bar\Sigma:=n^{-1}\sum_{i=1}^nX_iX_i^\top$. 
Since $\wh{\Sigma}=\bar\Sigma-\bar{X}\bar{X}^\top$, we have
\ba{
\E\tr(\wh\Sigma)
=\tr(\E[\wh\Sigma])
=(1-1/n)\tr(\Sigma),
}
\ben{\label{hs-bound0}
\|\wh{\Sigma}-\Sigma\|_{H.S.}
\leq\|\bar{\Sigma}-\Sigma\|_{H.S.}+\|\bar{X}\bar{X}^\top\|_{H.S.}
}
and
\ba{
\sum_{j=1}^d|\Sigma_{jj}-\wh{\Sigma}_{jj}|
&\leq\sum_{j=1}^d|\Sigma_{jj}-\bar{\Sigma}_{jj}|+\sum_{j=1}^d\bar X_j^2.
}
For any $j,k=1,\dots,d$, it holds that
\ba{
\E[|\bar\Sigma_{jk}-\Sigma_{jk}|^2]
&=\Var\left[\frac{1}{n}\sum_{i=1}^nX_{ij}X_{ik}\right]
\leq\frac{1}{n^2}\sum_{i=1}^n\E[X_{ij}^2X_{ik}^2].
}
Hence we have
\ben{\label{hs-bound}
\E\|\bar\Sigma-\Sigma\|_{H.S.}
\leq\sqrt{\frac{1}{n^2}\sum_{i=1}^n\E[|X_{i}|^4]}
}
and
\ba{
\E\left[\sum_{j=1}^d|\Sigma_{jj}-\bar{\Sigma}_{jj}|\right]
\leq\sum_{j=1}^d\sqrt{\frac{1}{n^2}\sum_{i=1}^n\E [X_{ij}^4]}.
}
Further, we have by the Jensen inequality
\ben{\label{hs-bound2}
\E\|\bar X\bar X^\top\|_{H.S.}=\E[|\bar X|^2]
=\frac{1}{n^2}\sum_{i=1}^n\E[|X_i|^2]
\leq\frac{1}{n}\sqrt{\frac{1}{n}\sum_{i=1}^n\E[|X_i|^4]}
}
and
\ba{
\sum_{j=1}^d\E[\bar X_j^2]
=\frac{1}{n^2}\sum_{j=1}^d\sum_{i=1}^n\E [X_{ij}^2]
\leq\frac{1}{n}\sum_{j=1}^d\sqrt{\frac{1}{n}\sum_{i=1}^n\E [X_{ij}^4]}.
}
Consequently, we obtain, using H\"older's inequality,
\ban{
\E[\delta_0^{1/3}(\wh{\Sigma},\Sigma)]
&\leq(\E\tr(\wh\Sigma)+\tr(\Sigma))^{1/6}(\E\|\wh\Sigma\|_{op}+\|\Sigma\|_{op})^{1/6}(\E\|\Sigma-\wh\Sigma\|_{H.S.})^{1/3}
\nonumber\\
&\leq C[\tr(\Sigma)]^{1/6}(\wh\delta+\|\Sigma\|_{op})^{1/6}\left(\frac{1}{n^2}\sum_{i=1}^n\E[| X_i|^4]\right)^{1/6}
\label{ef-delta0}
}
and
\ben{\label{ef-delta0p}
\E[\delta'^{1/2}_0(\wh{\Sigma},\Sigma)]\leq C\left(\sum_{j=1}^d\sqrt{\frac{1}{n^2}\sum_{i=1}^n\E [X_{ij}^4]}\right)^{1/2}.
}
Meanwhile, for any $p\geq2$, we have by the Jensen inequality
\ba{
|\bar{X}|^p
=\left\{\sum_{j=1}^d\left(\frac{1}{n}\sum_{i=1}^nX_{ij}\right)^2\right\}^{p/2}
\leq\left\{\sum_{j=1}^d\frac{1}{n}\sum_{i=1}^nX_{ij}^2\right\}^{p/2}
\leq\frac{1}{n}\sum_{i=1}^n|X_{i}|^p.
}
Hence we obtain
\ba{
\E\left[\sum_{i=1}^n|\wt X_i|^p\right]
\leq2^{p-1}\left(\sum_{i=1}^n\E[|X_i|^p]+n\E[|\bar{X}|^p]\right)
\leq2^p\sum_{i=1}^n\E[|X_i|^p].
}
Combining these bounds with \eqref{hs-bound} and \eqref{hs-bound2}, and using H\"older's inequality, we deduce
\ban{
\E[(\delta_1^*)^{1/4}]
&\leq C\left\{\frac{\E\|\wh\Sigma\|_{H.S.}}{n^2}\sum_{i=1}^n\E[| X_i|^4]
+\frac{(\E\|\wh\Sigma\|_{op})^{3/2}}{n^{3/2}}\sum_{i=1}^n\E[|X_i|^{3}]\right\}^{1/4}
\nonumber\\
&\leq C\left[\delta_1^{1/4}(W)
+\left(\frac{1}{n^2}\sum_{i=1}^n\E[| X_i|^4]\right)^{3/8}
+\left\{\frac{\wh\delta^{3/2}}{n^{3/2}}\sum_{i=1}^n\E[|X_i|^{3}]\right\}^{1/4}
\right]
\label{ef-delta1}
}
and
\ban{
\E[(\delta^*_2)^{1/3}]
&\leq C\left\{\frac{(\E\|\wh\Sigma\|_{op})^{1/2}}{n^{3/2}}\sum_{i=1}^n\E[|X_i|^{3}]
+\frac{1}{n^{2}}\sum_{i=1}^n\E[|X_i|^4]\right\}^{1/3}
\nonumber\\
&\leq C\left[\delta_2^{1/3}(W)
+\left\{\frac{\wh\delta^{1/2}}{n^{3/2}}\sum_{i=1}^n\E[|X_i|^{3}]\right\}^{1/3}\right].
\label{ef-delta2}
}
Note that we have
\ba{
\varkappa^{3/4}(\Sigma)\left(\frac{1}{n^2}\sum_{i=1}^n\E[| X_i|^4]\right)^{3/8}
&\leq\left(\varkappa^{3}(\Sigma)\frac{\|\Sigma\|_{H.S.}}{n^2}\sum_{i=1}^n\E[| X_i|^4]\right)^{3/8}\\
&\leq\left(\varkappa^{3}(\Sigma)\delta_1(W)\right)^{3/8}
\leq\left(\varkappa^{3}(\Sigma)\delta_1(W)\right)^{1/4},
}
where the first inequality follows from $\varkappa(\Sigma)^{-1}\leq\Lambda_1(\Sigma)=\|\Sigma\|_{H.S.}$ and the last one follows from \eqref{ef-wlog}. 
Thus, Theorem \ref{efron} follows from \eqref{ef1} and \eqref{ef-delta0}--\eqref{ef-delta2}. 

Next we prove Corollary \ref{coro-ef}. 
The first claim immediately follows from Theorems \ref{BALL2} and \ref{efron}. 
Besides, Since $\P(|W^*|>x\mid X)=1-\P(|W^*|\leq x\mid X)$ and $\P(|W|>q_n^*(\alpha))=1-\P(|W|\leq q_n^*(\alpha))$, the second claim follows from Theorems \ref{BALL2} and \ref{efron} along with Proposition 3.2 of \cite{Ko19}. 

\subsection{Proof of Proposition \ref{prop:op-norm}}\label{sec5.8}

Set $\bar\Sigma:=n^{-1}\sum_{i=1}^nX_iX_i^\top$. 
Since $\wh{\Sigma}=\bar\Sigma-\bar{X}\bar{X}^\top$, we have
\ba{
\E\|\wh\Sigma-\Sigma\|_{op}
\leq\E\|\bar\Sigma-\Sigma\|_{op}+\E\|\bar{X}\bar{X}^\top\|_{op}.
}
We have by \eqref{hs-bound2}
\[
\E\|\bar{X}\bar{X}^\top\|_{op}
\leq\E\|\bar{X}\bar{X}^\top\|_{H.S.}
=\frac{1}{n^2}\sum_{i=1}^n\E[|X_i|^2]
=\frac{\tr(\Sigma)}{n}.
\]
Since $L\geq1$, we complete the proof once we show
\ben{\label{aim:op-norm}
\E\|\bar\Sigma-\Sigma\|_{op}\leq C\left(L^2\sqrt{\frac{\|\Sigma\|_{op}\tr(\Sigma)}{n}}+L^4\frac{\tr(\Sigma)}{n}\right).
}
The proof of \eqref{aim:op-norm} is a trivial modification of that of Theorem 9.2.4 in \cite{Ve18}. 
First, note that \eqref{sub-gauss} is satisfied when we replace $X_i$ by $UX_i$ for any $d\times d$ orthogonal matrix $U$. Thus, without loss of generality, we may assume $\Sigma$ is a diagonal matrix. In addition, since $\Sigma_{jj}=0$ implies $\bar\Sigma_{jj}=0$, it suffices to consider the case that $\Sigma_{jj}>0$ for all $j=1,\dots,d$. 
Then, we have, with $Y_i:=\Sigma^{-1/2}X_i$,
\ba{
\|\bar\Sigma-\Sigma\|_{op}
&=\left\|\Sigma^{1/2}\left(\frac{1}{n}\sum_{i=1}^nY_iY_i^\top-I_d\right)\Sigma^{1/2}\right\|_{op}\\
&=\sup_{x\in\mathbb R^d:|x|\leq1}\left|x^\top\Sigma^{1/2}\left(\frac{1}{n}\sum_{i=1}^nY_iY_i^\top-I_d\right)\Sigma^{1/2}x\right|\\
&=\sup_{x\in T}\left|x^\top\left(\frac{1}{n}\sum_{i=1}^nY_iY_i^\top-I_d\right)x\right|,
}
where $T:=\{\Sigma^{1/2}x:x\in\mathbb R^d,|x|\leq1\}$. Hence we obtain
\ben{\label{rewrite-opnorm}
\|\bar\Sigma-\Sigma\|_{op}
=\sup_{x\in T}\left|\frac{1}{n}\sum_{i=1}^n(Y_i\cdot x)^2-|x|^2\right|
=\frac{1}{n}\sup_{x\in T}\left||Ax|^2-n|x|^2\right|,
}
where $A$ is the $n\times d$ matrix with rows $Y_i$. It is straightforward to check that $\max_{1\leq i\leq n}\|Y_i\cdot u\|_{\psi_2}\leq L|u|$ for all $u\in\mathbb R^d$. Therefore, We have by Theorem 9.1.3 and Exercise 8.6.6 in \cite{Ve18}
\ba{
\sqrt{\E\big[\sup_{x\in T}\left||Ax|-\sqrt{n}|x|\right|^2\big]}
\leq CL^2\E[\sup_{x\in T}|Z\cdot x|],
}
where $Z\sim N(0,I_d)$. Using the Schwarz inequality, we obtain
\ba{
\E[\sup_{x\in T}|Z\cdot x|]
=\E[\sup_{x\in \mathbb R^d:|x|\leq1}|\Sigma^{1/2}Z\cdot x|]
\leq\E|\Sigma^{1/2}Z|
\leq\sqrt{\tr(\Sigma)}.
}
Therefore, we deduce
\ba{
\sqrt{\E\big[\sup_{x\in T}\left||Ax|-\sqrt{n}|x|\right|^2 \big]}
\leq CL^2\sqrt{\tr(\Sigma)}
}
and
\ba{
\sqrt{\E\big[\sup_{x\in T}\left||Ax|+\sqrt{n}|x|\right|^2\big]}
&\leq CL^2\sqrt{\tr(\Sigma)}+\sqrt{\sup_{x\in T}(2\sqrt{n}|x|)^2}\\
&\leq CL^2\sqrt{\tr(\Sigma)}+2\sqrt n\sqrt{\|\Sigma\|_{op}}.
}
Combining these bounds with \eqref{rewrite-opnorm}, we obtain \eqref{aim:op-norm}. 

\subsection{Proof of Theorem \ref{wild} and Corollary \ref{coro-w}}

The proof is a straightforward modification of arguments in Section \ref{proof:efron} (replace $\widetilde{X}_i$ and $\wh \Sigma$ by $X_i$ and $\bar \Sigma$ respectively and remove all the computations involving $\bar X$) and therefore omitted. 

\appendix

\section{Appendix: A Result beyond Independence}

The proof of Theorem \ref{BALL2} can be modified, in a straightforward but tedious manner, to prove Gaussian approximation results on centered balls for sums of locally dependent random vectors. For example, we can obtain the following result for $m$-dependent sequences of random vectors (cf. \cite{HoRo48}).

\begin{theorem}\label{BALL3}
For integers $n>m\geq 0$,
let $\xi_1,\dots, \xi_n$ be a sequence of $m$-dependent random vectors in $\mathbb{R}^d$, that is, $\{\xi_1,\dots, \xi_i\}$ is independent of $\{\xi_{i+m+1},\dots, \xi_n\}$ for any $i$.
Let $W=\sum_{i=1}^n \xi_i$.
Suppose $\Var(W)=I_d$
and
\ben{\label{n50}
\E[|\xi_i|^6]\leq \delta^6
}
for any $i$ and a positive constant $\delta$.
Then we have, with $\wt m=m\vee 1$,
\bes{
&\sup_{r\geq 0}|\P(|W|\leq r)-\P(|Z|\leq r)|\\
\leq& C\left\{  \left(\frac{ [\wt m\sum_{i=1}^n \E[|\xi_i|^2]+n\wt m^3\delta^4][n\wt m^3\delta^4(n\wt m^3\delta^4+1)]}{d^{3}}\right)^{1/8} +\left(\frac{n\wt m^2\delta^3+n\wt m^3\delta^4}{d}\right)^{1/3}  \right\},
}
where $Z\sim N(0,I_d)$.
\end{theorem}

The bound above reduces to \eq{n42} under the setting of Corollary \ref{cor3} with the additional assumption that $\max_{1\leq i\leq n}\max_{1\leq j\leq d}\E[|X_{ij}|^6]\leq C$ (In this case: $\wt m=1, \delta\leq C\sqrt{d}/\sqrt{n}, \sum_{i=1}^n \E[|\xi_i|^2]=d$).
Note that the sixth moment assumption appears because we can no longer separate terms as in the proof of Theorem \ref{BALL2} without the independence assumption.

\begin{proof}[Proof of Theorem \ref{BALL3}]
We follow the proof of Theorem \ref{BALL2} and the notation used therein. 

We first introduce some new notation.
Let $[n]=\{1,\dots, n\}$.
For $u\in [n]$, let $A_u=\{v\in [n]: |v-u|\leq m\}$, so that $\xi_u$ is independent of $\{\xi_v: v\notin A_u\}$ by the $m$-dependence assumption.
Let $W^{(u)}=W-\sum_{v\in A_u}\xi_v$.
Similarly, for $u\in [n]$ and $v\in A_u$, let $A_{uv}=\{r\in [n]: |r-u|\leq m\ \text{or}\ |r-v|\leq m\}$.
Let $\eta^{(uv)}=\sum_{r\in A_{uv}}\xi_r$, $W^{(uv)}=W-\eta^{(uv)}$. Note that $\{\xi_u, \xi_v\}$ is independent of $W^{(uv)}$.

Note that for any random vectors $\zeta_1, \zeta_2\in \mathbb{R}^d$ independent of $W$, we have
\ben{\label{n51}
\E[(\zeta_1 \cdot W)(\zeta_2 \cdot W)|\zeta_1, \zeta_2]=\E[\zeta_1^\top W W^\top \zeta_2|\zeta_1, \zeta_2]=\zeta_1^\top \zeta_2.
}
Moreover, for any positive integer $k$ and vectors $\zeta_1, \dots, \zeta_k\in \mathbb{R}^d$, we have
\ben{\label{n52}
|\zeta_1|\cdots|\zeta_k|\leq |\zeta_1|^k+\dots+|\zeta_k|^k.
}
The condition \eq{n50} implies
\ben{\label{n55}
\E[|\xi_u|^k]\leq \delta^k,\quad \forall \ k=1,\dots, 6.
}

We need the following lemma which corresponds to Lemma \ref{portnoy} for the independent case. We will prove the lemma at the end of the Appendix.

\begin{lemma}\label{portnoy2}
For $u\in [n]$, $v\in A_u$ and $r\in A_{uv}$, we have
\ben{\label{n53}
\E[(\xi_r \cdot W^{(uv)})^2]\leq C\big\{ \E[|\xi_r|^2]+\wt m^2 \delta^4  \big\}
}
and
\ben{\label{n54}
\E[(\xi_u\cdot W^{(uv)})^2(\xi_v\cdot W^{(uv)})^2]\leq C\E[|\xi_u|^2|\xi_v|^2](n\wt m^3\delta^4+1).
}
\end{lemma}

Following the arguments leading to \eq{loo-1} and \eq{loo-2} but using the $m$-dependence assumption and Taylor's expansion around $W^{(u)}$ and $W^{(uv)}$, we obtain
\be{
\E [W\cdot \nabla f(W)]
=\sum_{u=1}^n\sum_{v\in A_u} \sum_{j,k=1}^d\E[\xi_{uj}\xi_{vk}]\E[\partial_{jk}f(W^{(uv)})]
+R_1'
}
and
\be{
\E\langle\Hess f(W), I_d\rangle_{H.S.}=\sum_{u=1}^n\sum_{v\in A_u} \sum_{j,k=1}^d\E[\xi_{uj}\xi_{vk}]\E[\partial_{jk}f(W^{(uv)})]
+R_2',
}
where
\[
R_1'=\sum_{u=1}^n\sum_{v\in A_u}\sum_{j,k=1}^d\E[\xi_{uj}\xi_{vk} (\partial_{jk}f(W^{(u)}+U\sum_{r\in A_u}\xi_r)-\partial_{jk}f(W^{(uv)})],
\]
\[
R_2'=\sum_{u=1}^n\sum_{v\in A_u}\sum_{j,k=1}^d\E[\xi_{uj}\xi_{vk}]\E[\partial_{jk}f(W)-\partial_{jk}f(W^{(uv)})],
\]
and $U$ is a uniform random variable on $[0,1]$ independent of everything else.

For ease of notation, we assume $U=1$ in the above $R_1'$. It will be easy to see that the final bound on $|R_1'|$ holds for any $U\in [0,1]$.
Straightforward modifications of \eq{r1-decomp} and the arguments after that in the proof of Theorem \ref{BALL2}, we obtain
\be{
R_1'= R_{11}'+R_{12}'+R_{13}'+R_{14}',
}
where
\be{
|R_{11}'|\leq C\eps^{-2} \sum_{u=1}^n \sum_{v\in A_u} \E\left[|\xi_u\cdot \eta^{(uv)}||\xi_v\cdot \eta^{(uv)}|\right],
}
\be{
|R_{12}'|\leq C\eps^{-2} \sum_{u=1}^n \sum_{v\in A_u} \E\left[ |\xi_u\cdot \eta^{(uv)}|(|\xi_v\cdot W^{(uv)}|+|\xi_v\cdot Z|)+|\xi_v\cdot \eta^{(uv)}|(|\xi_u\cdot W^{(uv)}|+|\xi_u\cdot Z|)   \right],
}
\bes{
|R_{13}'|\leq C\eps^{-3} \sum_{u=1}^n \sum_{v\in A_u}\E\Big[ &\left( |\xi_u\cdot W^{(uv)}|+|\xi_u\cdot Z|\right) \left(|\xi_v\cdot W^{(uv)}|+|\xi_v\cdot Z| \right)\\
&\times\left(|\eta^{(uv)}\cdot W^{(uv)}|+|\eta^{(uv)}\cdot Z|+|\eta^{(uv)}|^2 \right)\Big],
}
\be{
|R_{14}'|\leq C\eps^{-2} \sum_{u=1}^n \sum_{v\in A_u} \E \Big[ |\xi_u\cdot \xi_v|\left(|\eta^{(uv)}\cdot W^{(uv)}|+|\eta^{(uv)}\cdot Z|+|\eta^{(uv)}|^2 \right) \Big].
}
Using \eq{n52} and \eq{n55}, we have
\be{
|R_{11}'|\leq C \eps^{-2} n\wt m^3\delta^4.
}
For any $r\in A_{uv}$, 
let $A_{uvr}=\{s\in [n]: |s-u|\wedge |s-v|\wedge |s-r|\leq m\}$
and let $W^{(uvr)}=W-\sum_{s\in A_{uvr}}\xi_s$.
Note that $\{\xi_u, \xi_v, \xi_r\}$ is independent of $W^{(uvr)}$ by the $m$-dependence assumption.
We have
\bes{
\E[|\xi_u\cdot \xi_v||\xi_r\cdot W^{(uv)}|]\leq& \E[|\xi_u\cdot \xi_v||\xi_r\cdot (W^{(uv)}-W^{(uvr)})|]+\E[|\xi_u\cdot \xi_v||\xi_r\cdot W^{(uvr)}|]\\
\leq & C\wt m \delta^4+ C (\wt m \delta^4+\delta^3),
}
where we used \eq{n51}--\eq{n55}.
In addition, conditional on $\xi_u$, $\xi_u\cdot Z$ follows the normal distribution with mean 0 and variance $|\xi_u|^2$. Hence we obtain
\be{
|R_{14}'|\leq C \eps^{-2} n(\wt m^2\delta^3+\wt m^3\delta^4).
}
Similarly, 
\be{
|R_{12}'|\leq C \eps^{-2} n(\wt m^2\delta^3+\wt m^3\delta^4).
}
A main term in the upper bound for $|R_{13}'|$ is 
\be{
C\eps^{-3} \sum_{u=1}^n \sum_{v\in A_u} \sum_{r\in A_{uv}}\E[|\xi_u\cdot W^{(uv)}||\xi_v\cdot W^{(uv)}||\xi_r\cdot W^{(uv)}|],
}
which is bounded by, for any $\theta>0$,
\bes{
&C\eps^{-3} \sum_{u=1}^n \sum_{v\in A_u} \sum_{r\in A_{uv}}\left\{ \frac{\E[|\xi_r\cdot W^{(uv)}|^2]}{\theta}+\theta \E[(\xi_u\cdot W^{(uv)})^2(\xi_v\cdot W^{(uv)})^2]   \right\}\\
\leq &C\eps^{-3} \sum_{u=1}^n \sum_{v\in A_u} \sum_{r\in A_{uv}}\left\{ \frac{\E[|\xi_r|^2]+\wt m^2 \delta^4}{\theta}+\theta \delta^4(n\wt m^3\delta^4+1)  \right\}\\
\leq & C\eps^{-3}\wt m^2 \left\{\frac{\sum_{r=1}^n \E[|\xi_r|^2]+n\wt m^2\delta^4}{\theta}+\theta n\delta^4 (n\wt m^3\delta^4+1)   \right\},
}
where the first inequality follows from Lemma \ref{portnoy2}.
By optimizing over $\theta$, the bound becomes
\be{
C\eps^{-3 }\sqrt{[\wt m\sum_{i=1}^n \E[|\xi_i|^2]+n\wt m^3\delta^4][n\wt m^3\delta^4(n\wt m^3\delta^4+1)]}.
}
By similar and easier argument for other terms, we obtain 
\be{
|R_{13}'|\leq C\eps^{-3 }\sqrt{[\wt m\sum_{i=1}^n \E[|\xi_i|^2]+n\wt m^3\delta^4][n\wt m^3\delta^4(n\wt m^3\delta^4+1)]}.
}
Therefore, 
\be{
|R_1'|\leq C \eps^{-2} n(\wt m^2\delta^3+\wt m^3\delta^4)+C\eps^{-3 }\sqrt{[\wt m\sum_{i=1}^n \E[|\xi_i|^2]+n\wt m^3\delta^4][n\wt m^3\delta^4(n\wt m^3\delta^4+1)]}.
}
We can prove by a similar argument to the above that $R_2'$ has the same bound.
Combining with \eq{eq:smoothing}, we obtain the theorem.
\end{proof}

\begin{proof}[Proof of Lemma \ref{portnoy2}]
We first prove \eq{n53}. 
Recall $A_{uvr}=\{s\in [n]: |s-u|\wedge |s-v|\wedge |s-r|\leq m\}$
and $W^{(uvr)}=W-\sum_{s\in A_{uvr}}\xi_s$.
We have
\be{
\E[(\xi_r \cdot W^{(uv)})^2]\leq 2 \E\big[(\xi_r\cdot (W^{(uv)}-W^{(uvr)}))^2 +(\xi_r\cdot W^{(uvr)})^2   \big].
}
From \eq{n52} and \eq{n55}, we have
\be{
\E[(\xi_r\cdot (W^{(uv)}-W^{(uvr)}))^2]\leq C \wt m^2\delta^4.
}
Moreover, because $\xi_r$ is independent of $W^{(uvr)}$, we have, from \eq{n51},
\be{
\E[(\xi_r\cdot W^{(uvr)})^2]\leq \E[|\xi_r|^2]+ C\wt m^2\delta^4.
}
This proves \eq{n53}.

Next, we prove \eq{n54}. Because $\{\xi_u, \xi_v\}$ is independent of $W^{(uv)}$, we can treat $\xi_u=a$ and $\xi_v=b$ as non-random vectors in the proof. By considering all the non-zero expectation terms, we obtain
\bes{
\E[(a\cdot W^{(uv)})^2(b\cdot W^{(uv)})^2]=&\E[\sideset{}{'}\sum (a\cdot \xi_r)(a\cdot \xi_s)(b\cdot \xi_p)(b\cdot \xi_q)]\\
&+\E[\sideset{}{''}\sum (a\cdot \xi_r)(a\cdot \xi_s)(b\cdot \xi_p)(b\cdot \xi_q)],
}
where the first sum $\sideset{}{'}\sum$ is over indices $r,s,p,q\notin A_{uv}$ that are all connected, that is, each index is within distance $m$ from one of the other three indices and the second sum $\sideset{}{''}\sum$ is over indices $r,s,p,q\notin A_{uv}$ that can be separated into two disconnected parts, e.g., $\{r,s\}$ and $\{p,q\}$ such that $|r-s|\leq m$, $|p-q|\leq m$, but $\{r,s\}$ and $\{p,q\}$ are at least distance $m+1$ apart.

Because there are at most $C n \wt m^3$ terms in the first sum $\sideset{}{'}\sum$, we have, from \eq{n52} and \eq{n55},
\be{
\E[\sideset{}{'}\sum (a\cdot \xi_r)(a\cdot \xi_s)(b\cdot \xi_p)(b\cdot \xi_q)]\leq C|a|^2|b|^2 n\wt m^3 \delta^4.
}
A typical term in the second sum $\sideset{}{''}\sum$ is
\ben{\label{n56}
\E\Big\{\sum_{r,s\notin A_{uv}: \atop |r-s|\leq m}(a\cdot \xi_r)(a\cdot \xi_s)  \E\big[\sum_{p,q\notin A_{uv}\cup A_{rs}:\atop |p-q|\leq m}(b\cdot \xi_p)(b\cdot \xi_q)\big] \Big\}.
}
From \eq{n51},
\be{
\E[\sum_{p,q\notin A_{uv}\cup A_{rs}:\atop |p-q|\leq m}(b\cdot \xi_p)(b\cdot \xi_q)]=|b|^2+O(|b|^2\wt m^2 \delta^2).
}
Hence, \eq{n56} can be simplified as
\bes{
&\E[\sum_{r,s\notin A_{uv}: \atop |r-s|\leq m}(a\cdot \xi_r)(a\cdot \xi_s) ] |b|^2+O(|a|^2|b|^2 n \wt m^3\delta^4)\\
=&|a|^2|b|^2+ O(|a|^2|b|^2 \wt m^2 \delta^2) +O(|a|^2|b|^2 n \wt m^3\delta^4),
}
where we used \eq{n51} again in the last equation. Note that 
\be{
2 \wt m^2\delta^2\leq  1+\wt m^4 \delta^4   \leq 1+n \wt m^3 \delta^4 .
}
Therefore, \eq{n56} is bounded by
\be{
C|a|^2|b|^2(n \wt m^3 \delta^4+1).
}
All the other terms in the second sum $\sideset{}{''}\sum$ have the same upper bound.
This proves \eq{n54}.
\end{proof}

\section*{Acknowledgements}

We thank Wei Biao Wu for pointing us to the reference \cite{XuZhWu19}.
Fang X. was partially supported by Hong Kong RGC ECS 24301617 and GRF 14302418 and 14304917, a CUHK direct grant and a CUHK start-up grant. 
Koike Y. was partially supported by JST CREST Grant Number JPMJCR14D7 and JSPS KAKENHI Grant Numbers JP17H01100, JP18H00836, JP19K13668.


\end{document}